\documentclass[11pt]{amsart}
\usepackage{amssymb, amsfonts}
\usepackage{amsmath,amsthm}
\usepackage{xcolor}
\usepackage{appendix}
\usepackage{hyperref}  

\voffset = - 0.1 cm
\newtheorem{theorem}{Theorem}[section]
\newtheorem{lemma}[theorem]{Lemma}
\newtheorem{proposition}[theorem]{Proposition}
\newtheorem{corollary}[theorem]{Corollary}
\theoremstyle{definition}
\newtheorem{remark}[theorem]{Remark}
 
\newcommand{\beq}{\begin{equation}}
\newcommand{\eeq}{\end{equation}}
\newcommand{\bal}{\begin{aligned}}
\newcommand{\eal}{\end{aligned}}
\newcommand{\ben}{\begin{enumerate}}
\newcommand{\een}{\end{enumerate}}
\newcommand{\beqr}{\begin{eqnarray*}}
\newcommand{\eeqr}{\end{eqnarray*}}

\def\P{{\mathbb P}} 
\def\R{{\mathbb R}} 

\def\E{{\mathbb E}}
\def\Q{{\mathbb Q}} 
\def\N{{\mathbb N}}

\def\eps{{\varepsilon}}
\newcommand{\mathd}{\mathrm{d}}

\DeclareMathOperator{\Dom}{Dom}
\DeclareMathOperator{\Tr}{Tr}
\DeclareMathOperator{\Law}{Law}
\DeclareMathOperator{\Ker}{Ker}

\begin{document}  

\title{ Exponential integrability of the solution 
to the stochastic Burgers equation driven by  
white noise   }
  
\author{Francesco C. De Vecchi}
\address{Università degli Studi di Pavia, Dipartimento di Matematica ``Felice Casorati'', Via Adolfo Ferrata 5, 27100 Pavia PV}
\email{francescocarlo.devecchi@unipv.it}

\author{Josef Jan\'ak}
\address{Università degli Studi di Pavia, Dipartimento di Matematica ``Felice Casorati'', Via Adolfo Ferrata 5, 27100 Pavia PV}
\email{josefjanak@seznam.cz}

\author{Enrico Priola}
\address{Università degli Studi di Pavia, Dipartimento di Matematica ``Felice Casorati'', Via Adolfo Ferrata 5, 27100 Pavia PV}
\email{enrico.priola@unipv.it}
 
 \begin{abstract} We  study stochastic Burgers equation driven by  a rough noise $(-\Delta)^{\gamma} dW_t$, where $\Delta$ is the Laplacian in one dimension with Dirichlet boundary conditions, and $\gamma \in [0,1/4)$. We prove exponential estimates for  the solution $X_t^x$, starting from $x \in L^2(0,1)$, by showing that there exists some constant $\lambda >0$ for which
\begin{equation} \label{ds}  
\mathbb{E}
	\left[\exp\left(\lambda \sup_{t\in[0,T]}\|X_t^x\|_{L^2(0,1)}^2 \right) \right]< \infty.
\end{equation} 
This estimate was   known only in the case of trace class noise  when  $-1/2 <\gamma < -1/4 $ since in that case  one can  use the It\^o formula. To prove \eqref{ds}  we combine the  Bou\'e-Dupuis method with an argument used in  
 [Da Prato-Debussche, Potential Anal. 2007]. Estimates \eqref{ds} have important applications in large deviation theory, among others.
 We also deduce a  new Lipschitz regularizing effect for the corresponding   Markov semigroup.   
 \end{abstract}  
   
\maketitle

\noindent {\small {\it Keywords:} 
Stochastic Burgers equation; space-time white noise; exponential integrability, 
Lipschitz regularizing effect. 

}

\section{Introduction }

In this paper we are considering the following (Burgers) stochastic partial differential equations (SPDEs)
\begin{equation} \label{eq:Burgers} 
	\begin{cases}
		\mathd X_{t} &= \Delta X_{t} \, \mathd t + \frac{1}{2} \partial_\xi \left( X^2_{t} \right) \, \mathd t + (- \Delta)^{\gamma} \, \mathd W_{t}, \\
		X_0 & = x, 		\\
		X_{t}|_{\partial \Lambda} & = 0, \quad 0 < t \leq T,
	\end{cases} 
\end{equation}  
where $\Lambda=(0,1)$, $W$ is a cylindrical Brownian motion with values in $H := L^2(\Lambda)$ (namely $\mathd W_t(\xi)$ is the Gaussian white noise on $[0,T] \times \Lambda$), the initial value $x $ is taken in $H$ (and it is supposed to be deterministic), and $\gamma \in [0, 1/4)$ is a constant. 
 Equation \eqref{eq:Burgers} is much investigated (see \cite{BFZ2026,BCM2026, DPDlincei1998,DPD,DPDT, G, JP, P21, WWLHaranck2011}  and the references therein); it is also a simplified model of turbulence; see, e.g., \cite{Burgers}. 
   
The condition $\gamma <\frac{1}{4}$ is needed to guarantee that the solution $X_t$ belongs to a space of functions. Since, in the case $\gamma \ge \frac{1}{4} $ the solution is a distribution and some renormalization process is needed to define the nonlinearity in the equation (see, e.g., \cite{ MR3592748,MR3071506} for the case $\gamma=\frac{1}{2}$ and \cite{MR1941997} for a case similar to $\gamma=\frac{1}{4}$); the more irregular regimes will be subject to future investigation. 
 
The main aim of the paper is to prove some exponential estimates for 
the solution to the Burgers equation \eqref{eq:Burgers}, namely we prove that there is some constant $\lambda >0$ such that 
\begin{equation}\label{eq:exponentialboundintro}
	\mathbb{E}\left[\exp\left(\lambda \sup_{t\in[0,T]}\|X_t\|_{L^2(\Lambda)}^2 \right) \right] = c_T < \infty, 
\end{equation} 
(see Theorem \ref{theorem:exponentialbound} for a more precise statement of the main result). We also obtain from \eqref{eq:exponentialboundintro}  exponential integrability of the invariant measure (see Theorem \ref{theorem:expBurgersinvariant}). Extensions of   \eqref{eq:exponentialboundintro} are given  in Lemma \ref{deri} for the derivative process of $X_t^x = X_t$ with respect to the initial condition $x$. These results lead to a new Lipschitz regularizing effect for the Markov semigroup $(P_t)$ corresponding to \eqref{eq:Burgers}; see Theorem  \ref{theorem:Lip}.

\noindent Estimate \eqref{eq:exponentialboundintro} is known only in the case of trace class noise, in particular   when  $-1/2 <\gamma < -1/4 $. In this situation one can also use the It\^o formula
to derive \eqref{eq:exponentialboundintro} and with this method 
the constant $c_T$  depends also on the trace of $(-\Delta)^{\gamma}$ (cf. \cite{DP} and \cite{DPDlincei1998}). We also mention the recent paper  \cite{BFZ2026} where in Section 5.2   exponential estimates like \eqref{eq:exponentialboundintro} are used  to show  large deviations  upper bound  for invariant measures of stochastic Burgers equations driven by trace class noises depending on $\epsilon.$   
We concentrate on the case $\gamma \in [0,1/4)$ which is the most difficult case we can consider. However we could also consider $\gamma \in (-1, 1/4)$. In this general case the only  change would be   related to Theorem \ref{theorem:Lip}: in  the estimate \eqref{eq:LipP} the square root $\sqrt{t}$ should replaced by $t^{\frac{1}{2}-(\gamma \wedge 0)}$.

Establishing exponential bounds of the form \eqref{eq:exponentialboundintro} can be useful (and sometimes essential) for the study of some properties of SPDEs such as (here we mainly cite  only some references related to stochastic Burgers and Navier-Stokes equations): exponential estimates for the invariant measures to SPDEs (see, e.g., Chapters 5 and 6 of \cite{DP}) and application to large deviation principle for these measures \cite{BFZ2026,BC2017,CP2022}, study of the log-Harnack inequality \cite{WWLHaranck2011},  study of infinite dimensional Hamilton-Jacobi-Bellman equations \cite{DDsicon,DPDlincei1998}, and the study of the rate of convergence of numerical approximations \cite{BCM2026,HJ2020}.
 
Recall that the more generic strategy to prove  bounds of the form \eqref{eq:exponentialboundintro} is the use of Lyapunov functions, which works when the driving Gaussian noise is of trace class and thus the infinite dimensional It\^o formula can be applied to the function $\exp(\lambda\|x\|_{L^2(\Lambda)}^2)$ (see, e.g., Proposition 5.10 \cite{DP} for Burgers equation with trace class  noise). 
\noindent The same strategy cannot be directly applied when the noise is more singular as the one studied here. Before describing our method, let us mention that, in the setting of singular SPDEs (when the solution to a nonlinear SPDE is a distribution and renormalization is needed for the well-posedness of the equation), some exponential type estimates for the expectation of the solution has been proved.  More precisely, in  \cite{GH2021,MW2020} (the case of stochastic quantization) and \cite{HZ2025} (the case of stochastic Navier-Stokes equation), the authors prove exponential moment estimates of the form $ \mathbb{E}\left[ \exp( \lambda \|X_t\|^{\alpha})\right ] < +\infty $ for a suitable norm $\| \cdot  \|$ and some power $\alpha <1$ (indeed $\alpha <1$ in  \cite{GH2021,MW2020}   and $\alpha=1/2$ in \cite{HZ2025}). We also mention that
  \cite{HS2022}   proves the, conjectured, best possible exponential moment  estimate available in the stochastic quantization case, exploiting, in an essential way, the superlinear dissipative term of the stochastic quantization equation. Since no superlinear  dissipative term is present in the stochastic Burgers equation, their technique cannot be applied (at least, to the best of our knowledge and in a trivial way) to equation \eqref{eq:Burgers} studied in the present article.

A novelty of the current paper consists in the method used for proving the bound \eqref{eq:exponentialboundintro} which is based the combination of two ideas: the application of Bou\'e-Dupuis formula (see \cite{BD}) to SPDEs, and the use of a Da Prato-Debussche trick which involves  an Ornstein Uhlenbeck type process depending on a random parameter $\alpha$ (see \cite{DPD}). In particular in \cite{DPD}  polynomial integrability  estimates for  the solution $X_t$ of  \eqref{eq:Burgers} have been established in the   white noise case of $\gamma =0.$
 
The Bou\'e-Dupuis formula, firstly proposed in \cite{BD}, provides a connection between the expectation of the exponential of (generic) functional  of Brownian motion and a variational problem with respect to the same Brownian motion shifted through an $L^2$-predictable process $u$ (see Proposition \ref{proposition:BD}). This formula, among many other applications, has been used to study large deviations for SPDEs (see, e.g., \cite{Dupuis2000,BDM2008,CP2022,LRZ2013,RZZ2010}). 

In the current paper we instead exploit the Bou\'e-Dupuis formula in order to obtain bounds on the expectations of the form \eqref{eq:exponentialboundintro}, by proving some a-priori estimates on the solution $X^{n,u}$ of the Galerkin approximation \eqref{eq:Burgersapporximation} to Burgers equation \eqref{eq:Burgers} with an additional $L^2$-predictable drift $u$. To the best of our knowledge, apart from the related (but not identical) use of the Bou\'e-Dupuis formula for studying the Gibbs measures of quantum field theory and stochastic quantization (see, e.g., \cite{BG2020,barashkov2022stochastic,barashkov2022invariant,BG2023,BGH2023}), it seems that the present article represents the first use of Bou\'e-Dupuis methods for bounding expectations of the solution to SPDEs as an alternative to 
Lyapunov function methods.

In order to obtain  the previous  a-priori estimates on the solution $X^{n,u}$ of the Galerkin approximation \eqref{eq:Burgersapporximation} we use the 
 second main idea of the paper which is obtained from  \cite{DPD}. 
  We briefly give a sketch of  this  method. In \cite{DPD} when  $\gamma =0$, for every  $\alpha\ge 0$ they consider  the $L^2(\Lambda)$-valued Ornstein-Uhlenbeck type process    
$$   
z_{\alpha}(t)=\int_0^t e^{(A-\alpha)(t-s)}\mathd W_s,\;\; t \ge 0, 
$$
 depending on $\alpha \ge 0$ (cf. \eqref{zz}). They look for  the equation verified by $Y_t^{\alpha} =X_t -z_{\alpha}(t) $. After establishing  a version of $z_{\alpha}$ which depends in a continuous way also on $\alpha$ (cf. Proposition \ref{continuo}) one can consider in the equation verified by $Y_t^{\alpha}$, a random  $ \alpha (\omega)$ such that, for any $\omega$, $\P$-a.s.,  
 $$
\sup_{t \in [0,T]}  \| z_{\alpha(\omega)}(t, \omega)\|_{L^4(\Lambda)} \leq 1.
 $$ 
 This allows to obtain by the Gronwall lemma an  estimate involving    $\E [\|X_t\|^2_{L^2(\Lambda)}]$ (also other moments estimates can be obtained).   
  We use a similar idea  with the Galerkin approximation (see the proof of Theorem \ref{theorem:apriorigamma} and in particular the  argument given after formula \eqref{vs}). 
    
 The plan of the paper is as follows. In Section 2  we fix notation and provide preliminary results useful to study  the stochastic Burgers equation.
   We also recall in Section \ref{section:BD} the  Bou\'e-Dupuis formula.
   Section  \ref{section:Burgers} presents some results on  stochastic Burgers equation, like the equivalence of weak and mild formulation   and the convergence of the Galerkin approximation to the solution of \eqref{eq:Burgers}.  We finish Section \ref{section:Burgers} with some   results on the process $z_{\alpha}$.\\
  Sections \ref{section:applicationBurgers} and \ref{section:exponentialbound}
  contain the main application of the Bou\'e-Dupuis formula to the Burgers equation.   The exponential integrability estimates are given in Section \ref{section:exponentialbound}. 
    After proving the exponential bounds of the form \eqref{eq:exponentialboundintro}, in Section \ref{section:applicationBurgers1}, we apply them to prove the exponential integrability of the $L^2$-norm with respect to the invariant measure of \eqref{eq:Burgers}.
    We also prove a new Lipschitz regularizing effect for 
the Markov semigroup $(P_t),$ associated with the solutions to the Burgers equation \eqref{eq:Burgers}:   
$$	|P_t\varphi(x)-P_t\varphi(x')|\leqslant\frac{L_{T,\varepsilon}\exp\left(\varepsilon\left(\|x\|_{L^2 (\Lambda)}^2+\|x'\|_{L^2(\Lambda)}^2\right)\right)}{\sqrt{t}}\|\varphi\|_{L^{\infty}} \|x-x'\|_{L^2(\Lambda)}.
 $$
$t \in (0,T]$, $x, x' \in L^2(\Lambda)$ with $\varphi \in B_b(H)$  (see Theorem \ref{eq:LipP}).  This was known only in the case of trace class noise (see \cite{DPDlincei1998}). The result clearly implies the strong Feller property for $(P_t)$.  Apart from the two previous applications of our exponential moment estimates, we  briefly discuss how to apply them   to the large deviations for invariant measures of  Burgers equations (see Remark \ref{remark:largedeviation}),  prove the log-Harnack inequality for the Burgers equation (see Remark \ref{remark:logHarnack}) and  study the Hamilton-Jacobi-Bellman equations associated with stochastic Burgers equations (see Section \ref{remark:HJB}). 
 An application  to weak existence for singular perturbations of stochastic Burgers equations is given Section \ref{girsanov} (this  uses the Girsanov theorem and our exponential estimate).\\
 \noindent Finally, Appendix A concerns with existence of  invariant measure and irreducibility of $(P_t)$. This together with the strong Feller property   leads to existence and uniqueness of invariant measure.

\section{Notation and some preliminary results }

\subsection{Notation}   

Consider $\Lambda = (0,1)$ and the real Hilbert space $H=L^2(\Lambda)$ equipped with the usual $L^2$-norm $\| \cdot \|=\|\cdot\|_{L^2_{\xi}} := \| \cdot \|_{L^2(\Lambda)}$ and the scalar product $\left\langle \cdot, \cdot \right\rangle := \left\langle \cdot, \cdot \right\rangle_{L^2_{\xi}}$. Even though the set $\Lambda$ is only one dimensional, we consider the standard Laplace operator with Dirichlet boundary conditions
$$
Au = \Delta u = \partial_{\xi \xi}^2 u = u''
$$
for a function $u$ which is regular enough (cf. \eqref{dom1} with $s=2$). $A$ is negative self-adjoint operator with the eigensystem $(\alpha_k, e_k)_{k=1}^{\infty}$:
\begin{equation} \label{sys3}
A e_k = - \alpha_k e_k = - \pi^2 k^2 e_k, \quad e_k(\xi) = \sqrt{2} \sin(k \pi \xi), \quad \xi \in [0,1], \, k \geq 1.
\end{equation}
To describe higher regularities, we consider for $s \geq 0$ the fractional Laplacian $(-A)^{s/2}$ and denote their domains by $H^s=H^s_{\xi} = H^s(\Lambda)$ with the norm $\| \cdot \|_s := \| \cdot \|_{H^s} := \|\cdot  \|_{H^s_{\xi}} := \| (-A)^{s/2} \cdot \|_{L^2_{\xi}}$. The Hilbert spaces 
\begin{equation} \label{dom1}
H^s = \Dom \left( (-A)^{s/2}) \right) = \Big\{ u \in L^2(\Lambda); \, \sum_{k=1}^{\infty} \alpha_k^s u_k^2 = \sum_{k=1}^{\infty} \alpha_k^s \left\langle u, e_k \right\rangle^2 < \infty \Big\}
\end{equation}
are equipped with the scalar product $\left\langle u, v \right\rangle_{H^s} = \sum_{k=1}^{\infty} \alpha_k^s u_k v_k$ for $u, v \in H^s$ ($H^0=H = L^2(\Lambda)$). By $B(0,n)$ we denote the closed ball in the space $H$ with the radius $n$, that is $B(0,n) = \{ u \in H; \, \| u \| \leq n \}$. Recall that $H^1 =H_0^1$ with equivalence of norms, where $H_0^1$ denotes  the space of all $u \in H$ with weak derivative $u' \in H$ and such that $u(\xi) = 0$ for $\xi \in \partial \Lambda$ (cf. Section 1.3 in \cite{Henry}, or \cite{Hairer}).
We denote by $(e^{At}, t \geq 0)= (e^{tA})$ the analytic semigroup on $H$ generated by $A$, i.e., in the present setting, the heat semigroup on $\Lambda$ with Dirichlet boundary conditions.

We will also use orthogonal projections with respect to $(e_k)$: 
\begin{equation*} \label{pp1}
 \begin{array} {l}
\pi_m=\sum_{k=1}^m e_k \otimes e_k, \;\;
\pi_m x = \pi_m(x) = \sum_{k=1}^m x_k e_k, \; \text{where $x_k = \langle x, e_k\rangle $, $x \in H$,}      
 \end{array}  
\end{equation*}
$m \ge 1.$ Since the eigenfunctions of the Laplacian are trigonometric functions, by standard results about $L^p$-convergence of Fourier series (see, e.g., Theorem 4.1.1 of \cite{Grafakos2014}), for any $1<p<+\infty$, we have that 
\begin{equation}\label{interpo}
\| \pi_n x \|_{L^p(\Lambda)} \le C_p \| x\|_{L^p(\Lambda)},\;\;\; x \in {L^p(\Lambda)},\; n \ge 1.
\end{equation} 
For $s \in \R$ and $p \geq 2$ we consider also $W^{s,p}(\Lambda)=W^{s,p}_{\xi}  := \{u \in L^p(\Lambda); \, \| u \|_{s,p} < \infty \}$, where $\| \cdot \|_{s,p} := \| (-A)^{s/2} \cdot \|_{L^p(\Lambda)}$. These are fractional Sobolev spaces associated to $A$ defined as Bessel potentials spaces. The space $C(\bar{\Lambda}) := C_{\xi}$ of all real continuous functions on $\bar{\Lambda}$ is equipped with the supremum norm $\| \cdot \|_{C(\bar{\Lambda})}$. The spaces $C^{k, \alpha}(\Lambda)$ for $k \geq 0$ integer and $0 < \alpha < 1$ are H\"older spaces equipped with usual norms. In the case $k=0$ we simply write $C^{\alpha}(\Lambda)$ or $C^{\alpha}_{\xi}$.

\begin{remark}\label{remark:sobolevebmedding}  
		For the Sobolev spaces $H^s_{\xi}$ and $W^{s,p}_{\xi}$ the Sobolev embedding theorems holds. Namely if $s_1-\frac{1}{p_1}>s_2-\frac{1}{p_2}$, for some $s_1,s_2\geq 0$ and $p_1,p_2\in[1,+\infty)$, then $W^{s_1,p_1}_{\xi}$ is compactly embedded in $W^{s_2,p_2}_{\xi}$ (see Chapter 7 in \cite{AdamsFournier} or  Section 3.3.1 of    \cite{TriebelI}); here we are using the convention that 
		$W^{s,2}_{\xi}:=H^s_{\xi}$. We recall 
	the  interpolation properties of the spaces $H^s_{\xi}$, namely: let $s_1<s_2\in\R$ and let $\theta \in (0,1)$ then the interpolation space of $H^{s_1}_{\xi}$ and $H^{s_2}_{\xi}$ is exactly $H^{\theta s_1 +(1-\theta)s_2}_{\xi}$ and for any $f \in H^{s_2}_{\xi}$ we have 
	\[\|f \|_{H^{\theta s_1+(1-\theta)s_2}_{\xi}} \leqslant C \|f \|_{H^{s_1}_{\xi}}^{\theta} \|f \|_{H^{s_2}_{\xi}}^{(1-\theta)},\]
	(see, e.g., Section 3.3.6 of \cite{TriebelI}). \qed
\end{remark} 
Throughout this work we fix a finite time horizon $T > 0$ (that could be arbitrarily large) and we write  $L^p_{t}:=L^p([0,T])$ for the standard $L^p$ spaces, with respect to Lebesgue measure, on $[0,T]$. Furthermore, if $s\in(0,1)$ we let $H^s_{t}:=H^s([0,T])$  be the Hilbert space with the following norm 
\[
\|f\|_{H^s_t}^2:=\|f\|_{L^2_t}^2+  \int_0^T{\frac{\|\Delta_h f\|_{L^2_{t}}^2}{h^{1+2s}}\mathd h}\]
where $\Delta_h f(t) = f(t + h) - f(t)$ when both $t, t+h \in [0,T]$ and $0$ otherwise. We denote by $H^1_t:=H^1([0,T])$ the Sobolev space on $[0,T]$ of functions having weak derivatives in $L^2$, equipped by the standard norm. We are going to consider space of functions, of various regularity, defined on the interval $[0,T]$ and taking values in some Banach space of functions with respect the space variable $\xi\in\Lambda$, such as $L^p(\Lambda), H^s(\Lambda),\ldots $. For instance, for   $1 \le  p < \infty$, $L^p([0,T]; H^s(\Lambda))$ is the set of all functions $u$ with values in $H^s(\Lambda)$ such that $u$ is strongly measurable on $[0,T]$ and $\| u(t)\|_{H^s_{\xi}}^p$ is integrable on $[0,T]$.  
 
We denote these spaces by $A_tB_{\xi}$ where $A,B=L^p,H^s,C^{\alpha}$. For example $C^{\alpha}_tH^s_{\xi}:=C^{\alpha}([0,T];H^{s}(\Lambda))$, $\alpha \in (0,1)$. 
 When $\alpha =0$, $C([0,T];H^{s}(\Lambda)):=C^{0}([0,T];H^{s}(\Lambda))$ indicates the space of  continuous functions from $[0,T]$ with values in $H^{s}(\Lambda).$

The norm of these spaces $A_tB_{\xi}$ is the natural generalization of the norms of $L^p_t,H^s_t,C^{\alpha}_t$ where the absolute values of $\R$ is replaced by the norm of $B_{\xi}$. 

\vskip 1mm
We denote by
${B}_b(H)$
  the Banach space of all real, bounded and  Borel functions on
  $H$  endowed with the supremum norm $\| f \|_{L^{\infty}} = \sup_{x \in H} |f(x)|$, $f \in {B}_b(H).$ Moreover 
  $C_b(H) \subset  B_b(H)  $ indicates the subspace of all   bounded and continuous  functions. We also define
 $C^{1} (H) \subset C_b(H)$ as  the  subspace of all bounded  functions which are Fr\'echet differentiable on $H$ with 
 a continuous  Fr\'echet derivative    on $H$ and $C^1_b(H) \subset C^1(H)$ as the subspace of all functions with bounded Fr\'echet derivative on $H.$

\begin{remark}\label{remark:sobolevebmedding2}
	The same Sobolev embedding theorems mentioned in Remark \ref{remark:sobolevebmedding} holds for the spaces $H^s([0,T])$, $L^p([0,T])$ and $C^{\alpha}([0,T])$. By combining the Sobolev embedding for the time-variable and the space-variable spaces one can applied them to the Banach spaces of the form $A_tB_{\xi}$ defined above. \qed
\end{remark}
  
\begin{remark}\label{remark:interpolation1}
	It is possible to generalize the interpolation result stated in Remark \ref{remark:sobolevebmedding} to the Lebesgue-Bochner spaces $L^p([0,T];H^s(\Lambda)):=L^p_tH^s_{\xi}$ described above. Indeed, since the spaces $H^s_{\xi}$ are separable, we can apply Theorem 1.1 of Chapter VII of \cite{LionsPeetre1964}, see also Theorem 5.1.2 in \cite{Bergh},
	and thus, if $p_1,p_2\in [1,+\infty]$, $s_1,s_2 \in \R$, $\theta \in [0,1]$, $\frac{1}{p_{\theta}}=\frac{\theta}{p_1}+\frac{1-\theta}{p_2}$, and $s_{\theta}=\theta s_1+(1-\theta) s_2$, we get that for any $f\in L^{p_1}_tH_{\xi}^{s_1}\cap L^{p_2}_tH_{\xi}^{s_2}$
	\[ \|f\|_{L^{p_{\theta}}_tH^{s_{\theta}}_{\xi}} \lesssim \| f\|_{L^{p_1}_tH_{\xi}^{s_1}}^{\theta} \|f \|_{L^{p_2}_tH_{\xi}^{s_2}}^{1-\theta} . \qed\]
\end{remark}

\begin{remark}\label{remark:heatregularization}
	In the paper we use the following two regularization properties of the heat kernel. 
		i) Consider a function $f \in L^2_tH^{s}_{\xi}$, where $s\in \R$, and let  $m_1,m_2 \geqslant 0$ such that 
	$m_1+ m_2<1$. If we define 
	\[f^H(t)=\int_0^te^{A (t-s)} f(s)\mathd s,\]
	then $f^H \in H^{m_1}_tH^{s+ 2m_2}_{\xi}$
	(see, e.g., Proposition 2.4.1 of Chapter VII of \cite{AmannvolumeII} for the $\R^n$-case). 
		ii) Let $f \in C_t^0 H^{s}_{\xi}$ then $f^H \in C^{0}_tH^{s+ 2 -\epsilon}_{\xi}$, for any $\epsilon>0$, $s \in \R$. This result can be  deduced from Proposition 5.9 in \cite{DPZ14} using  known estimates of the heat semigroup (see, for instance, formula (2.6) in \cite{DPDT}).
	\qed 
\end{remark}   
 
As we said throughout this work we fix a finite time horizon $T > 0$ and let $(\Omega, \mathcal F, (\mathcal F_t)_{t \geq 0}, \mathbb P)$ be a stochastic basis (satisfying the usual hypotheses) with a~cylindrical Wiener process $W= (W_t)$ on $H$ (cf. \cite{DPZ14}). The process $W$ is formally given by ``$W_t = \sum_{k=1}^{\infty} \beta_k(t) e_k$'' where $(\beta_k)$ are independent one di\-men\-sio\-nal Wiener processes adapted to the previous filtration; we are using the above orthonormal basis $(e_k)$ in $H$.

Throughout, all equalities and inequalities, unless otherwise mentioned, will be understood in the $\P$-a.s. sense. By $A \lesssim B $ we mean that there exists some constant $C > 0$ such that $A \leq C B$.

\subsection{Bou\'e-Dupuis formula and finite-dimensional SDEs}\label{section:BD}

In this section we give a brief summary of Bou\'e-Dupuis formula which is an  important mathematical tool used in this paper.

 Consider the following SDEs on a finite dimensional space $\R^n$
\begin{equation}\label{eq:BDSDE}
	\mathd X_t^{x,u}=b(X_t^{x,u})\mathd t+C u_t \mathd t +C dW_t,\quad X_0^{x,u}=x\in\R^n 
\end{equation}
where $C \in \R^{n \times n}$ is a matrix, and $b$ is a locally Lipschitz function, $u:=(u^1,...,u^n):[0,T] \times \Omega \rightarrow \R^n$ is  a predictable   process taking values in $\R^n$ such that \[\mathbb{E}\left[\sum_{i=1}^n \int_0^T |u^i_t|^2dt\right]=\mathbb{E}\left[\|u\|^2_{L^2([0,T];\R^n)}\right]<+\infty.\] We denote the space of such processes $u$ by $\mathcal{P}^{[0,T],\R^n}$.\\

Under the previous conditions on \eqref{eq:BDSDE}
 we have:

\begin{proposition}\label{proposition:BD}
Consider a measurable function $G:C^0([0,T];\R^n) \rightarrow \R_+$ and fix $x\in \R^n$. Suppose that for any  process $u\in \mathcal{P}^{[0,T],\R^n}$, equation \eqref{eq:BDSDE} has a unique strong solution which is global in time. Then we have
	\[\log \left(\mathbb{E}\left[\exp(G(X_{\cdot}^{x,0}))\right] \right) = \sup_{u \in \mathcal{P}^{[0,T],\R^n}}\mathbb{E}\left[G(X_{\cdot}^{x,u})-\frac{1}{2}\|u\|_{L^2([0,T];\R^n)}^2\right].\] 
\end{proposition}
\begin{proof}
	Proposition \ref{proposition:BD} is essentially Theorem 4.1 of \cite{BD} (see also Theorem 5.1 of \cite{BD} where the boundedness of the function $G$ is replaced by the fact that $G$ is bounded from below). The only difference in the statement of the present theorem and the one of Theorem 4.1 in \cite{BD} is that in the latter the coefficient $b$ of the SDE \eqref{eq:BDSDE} is assumed to be globally Lipschitz.
	This condition is there assumed only to assure the global in time existence of the solution to equation \eqref{eq:BDSDE}. Since here the global in time existence is assumed by hypotheses the current result is implied by the one in \cite{BD}.
\end{proof}

\section{Bou\'e-Dupuis formula and exponential integrability of Burgers equation}
 
\subsection{Stochastic Burgers equation}\label{section:Burgers}

In this paper we are going to consider the following Burgers stochastic partial differential equation  
\begin{equation}\label{eq:Burgersmain1}
	\mathd X_t^x=(AX_t^x+ \frac{1}{2}\partial_{\xi}(X_t^x)^2)\mathd t+(-A)^{\gamma}\mathd W_t,\quad X^x_0=x, 
\end{equation}
in the mild form    
\begin{equation} \label{sdd}
X_t^x=e^{At}x+\frac{1}{2}\int_0^te^{A (t-s)}\partial_{\xi}(X_s^x)^2 \mathd s+\int_0^te^{A (t-s)}(-A)^{\gamma}\mathd W_s,
\end{equation}
considering  the stochastic convolution process
\begin{equation} \label{eq:OUz}
z(t) := \int_0^t e^{A (t-s)} (-A)^{\gamma} \, dW_s, \;\; t \ge 0.
\end{equation}
Hence by a {\it  mild solution} to equation \eqref{eq:Burgersmain1}, we mean that for any $T>0$, there exists an $H$-valued adapted process $X^x = (X_t^x)_{t \in [0,T]}  $ with continuous paths  in $H$ such that  \eqref{sdd} holds (see also \cite{DPDT}).

\begin{remark} \label{sys}
Following Section 5.3 in \cite{DPZ14} one can  easily prove that 
$\P$-a.s. $z \in C([0,T]; H)$. To obtain this property it is essential that $\gamma < 1/4.$  Indeed we have 
\begin{align*}
\int_0^T t^{- 2 \alpha} \| e^{At} (-A)^{\gamma} \|_{HS}^2 \, dt &= \int_0^T t^{- 2 \alpha} \sum_{k=1}^{\infty} \alpha_k^{2 \gamma} e^{- 2 \alpha_k t} \, dt 
\lesssim 
\sum_{k=1}^{\infty} \frac{1}{k^{2 - 4 \alpha - 4 \gamma}} < \infty,
\end{align*}
if $\alpha < 1/4 - \gamma$.  Moreover,  by  means of the Kolmogorov test, as in Section 5.5.1 in \cite{DPZ14}, one can prove more, i.e., 
  $z \in C([0,T]; C(\bar{\Lambda}))$, $\P$-a.s. (therefore we will also have $z \in C([0,T]; L^4(\Lambda))$, $\P$-a.s.).
   
 To this purpose we only note that all the results of Section 5.5.1 in \cite{DPZ14} hold if
we deal with a positive diagonal operator $Q : H \to H$ (not necessarily bounded) such that $Q e_k = \lambda_k e_k$ (with  $\lambda_k>0$ and the basis $(e_k)$ which is defined in \eqref{sys3}). According to \cite{DPZ14} we have formally
$
\sqrt{Q}W_t $ $=\sum_{k=1}^{\infty} \sqrt{\lambda_k}\,  e_k \beta_k(t)
$
where $(\beta_k)$ are independent real Wiener processes (the series in general converges in a Hilbert space larger than $H$).
In our case $\sqrt{Q} = (-A)^{\gamma}$ and so we have formally
$$
(-A)^{\gamma} W_t = \sum_{k=1}^{\infty} \alpha_k^{\gamma} \, e_k \,  \beta_k(t),
\quad \text{i.e., $\; \lambda_k = \alpha_k^{2\gamma} = (\pi^2 k^2)^{2 \gamma}$.} \qed
$$ 
 \end{remark}

The mild formulation implies that  we can consider \eqref{eq:Burgersmain1} 
in the weak sense, namely, for any $f \in D(A) = H^2$,  
we have, $\P$-a.s., $t \in [0,T],$ 
\begin{equation}\label{eq:weaksolution}\langle X_t^x-x,f\rangle=
\int_0^t \langle \Delta_{\xi}f,X^x_s \rangle\mathd s
-\frac{1}{2}\int_0^t\langle\partial_{\xi}f,(X^x_s)^2 \rangle\mathd s+\int_0^t\int_{\Lambda}f(\xi)(-A)^{\gamma}\mathd W_s(\xi) \mathd \xi. 
\end{equation}
 Note that the last term can be written as $ \langle  (-A)^{\gamma} f, W_t\rangle $. 

We want to point out that, in the case of Burgers equation with solution process $X \in C^0_tL^2_{\xi}$ almost surely, the weak and mild formulations are equivalent.

\begin{proposition} \label{weakmild}
Suppose that  $X \in C^0_tL^2_{\xi}$ is an adapted process. Then $X$ is a weak solution to equation \eqref{eq:Burgersmain1} if and only if  $X$ satisfies equality \eqref{sdd}. 
\end{proposition} 

\begin{proof}  To prove that  the  mild formulation imply the  weak formulation, we can argue as in the proof of  Lemma 2.8 in  \cite{JP}. Let us prove the converse.   

By Theorem 5.4 of \cite{DPZ14}, the process $z$ defined in equation \eqref{eq:OUz} is the unique weak solution to the following linear SPDE
\[\mathd z(t)=Az(t)\mathd t+ (-A)^{\gamma}\mathd W_t,\quad z(0)=0.\]
This means that, if $X_t$ is a weak solution to equation \eqref{eq:Burgersmain1}, then the process $Y_t=X_t-z(t)$ is a weak solution to the following random PDE
\begin{equation}\label{eq:Yt1}
	\frac{\mathd Y_t}{\mathd t}=AY_t+\frac{1}{2}\partial_{\xi}\left(X_t^2\right),\quad Y_t=x,
\end{equation}  
where $Y$ is a process in $C^0_tL^2_{\xi}$. This means that, to prove that $X_t$ is a mild solution to equation \eqref{eq:Burgersmain1} is equivalent to prove that $Y_t$ is a mild solution to equation \eqref{eq:Yt1}. In order to obtain this last result let $f(t):=\frac{1}{2}\partial_{\xi}\left(X_t^2\right)$. With this notation, $Y_t$ solves an inhomogeneous PDE with linear part $A Y_t$ and affine part $f(t)$. Using that $L^1_{\xi} \subset H^{-\frac{1}{2}-\varepsilon}_{\xi}$ (where $\varepsilon$
is any strictly positive constant) we obtain that
the function $f \in 
C^0_t H^{-\frac{3}{2}-\varepsilon}_{\xi}$ . Since  $Y$ is a weak solution of a parabolic PDE in $L^2_{\xi}$, $Y$ is also a weak solution to the same parabolic PDE in $H^{-\frac{3}{2}-\varepsilon}_{\xi}$. Furthermore, the heat semigroup generated by $A$ is strongly continuous as a semigroup in $L^2_{\xi}$ and $H^{-\frac{3}{2}-\varepsilon}_{\xi}$. 
All these observations imply that we can apply the main theorem of \cite{Ball1977}, obtaining that $Y_t$ satisfies the equality 
\[Y_t=e^{A t}x+\int_0^t e^{A(t-s)}f(s)\mathd s\]
as a function belonging to $C^0_{t}H^{-\frac{3}{2}-\varepsilon}_{\xi}$. Let  $\varepsilon \in (0, 1/2)$. Since $x \in L^2_{\xi}$, and thus $e^{A t}x \in C^0_tL^2_{\xi}$, and, by (ii) in Remark \ref{remark:heatregularization}, $\int_0^{\cdot}e^{A(\cdot-s)}f(s)\mathd s \in C^0_tL^2_{\xi}$, we have proved the thesis. 
\end{proof}  
In the sequel we will also use the weak formulation of our SPDE.

 The proof of the next result is not difficult. Indeed one can follow \cite{G} or the strategy of the proof of  Theorem 14.2.4 in  \cite{DPZ96}  which treats the case   $\gamma=0.$ We give a sketch of proof in Appendix \ref{ppp2}. 
 \begin{proposition}\label{cia1}  
	For any $\gamma \in \left[0,\frac{1}{4}\right)$ and any $x \in H= L^2(\Lambda)$, there is a unique (mild) solution to equation \eqref{eq:Burgersmain1} (taking values in $C^0([0,T];H))$. 
\end{proposition} 
 
 We are going to consider also an approximation of equation \eqref{eq:Burgersmain1}. Namely, let $\pi_n$ be the  orthogonal projection onto the subspace $H_n$, which is the $n$ dimensional subspace of $H$ generated by the first $n$ eigenfunctions of the Laplacian $A$ with Dirichlet boundary conditions (cf. \eqref{pp1}). 
  We define, setting $x_n:=\pi_n(x),$
\begin{equation}\label{eq:Xapparoximation1}
	 \mathd X_t^{n,x_n}=(AX_t^{n,x_n}+\frac{1}{2}\pi_n\left(\partial_{\xi}(X_t^{n,x_n})^2\right))\mathd t+(-A)^{\gamma}\pi_n\mathd W_t,\quad X^{n,x_n}_0=x_n,    
\end{equation} 
where $x$ is a generic element of $H$. Equation \eqref{eq:Xapparoximation1} is a finite dimensional stochastic differential equation (SDE) defined on the space $H_n$. 
  
We write $z^n:=\pi_n(z)$ and $Y^{n,x_n}:=X^{n,x_n}-z^n_{}$. The process $Y^{n,x_n}$ solves the  following random ODE:
\begin{equation}\label{eq:Yapproximation1}  
	\frac{d}{dt}Y^{n,x_n}=AY^{n,x_n}_t+\frac{1}{2}\pi_n\left(\partial_{\xi}(Y^{n,x_n}_t+z^n_{}(t))^2\right),\;\; \quad Y^{n,x_n}_0=\pi_n(x). 
\end{equation}  
 Solving \eqref{eq:Yapproximation1} is equivalent to solve     
 \eqref{eq:Xapparoximation1}. A unique global solution   can be obtained since equation \eqref{eq:Xapparoximation1} has the structure of the following
   SDE on   $\R^n$:
\begin{equation}\label{easy}
	\mathd X_t=AX_t dt +   b(X_t)\mathd t +C dW_t,\quad X_0^{}=x\in\R^n 
\end{equation}
where $A, C \in \R^{n \times n}$ are matrices  and $b: \R^n \to \R^n$ is a locally Lipschitz function such that $b(x)\cdot x =0$, $x \in \R^n$. It is not difficult  to obtain a unique  global solution to \eqref{easy} by means of the Lyapunov function $|x|^2$, using the It\^o formula.

Using the previous notations we have the following result.  
 
\begin{theorem}\label{theorem:existenceconvergence}
	For any $\gamma \in \left[0,\frac{1}{4}\right)$ and any $x \in H= L^2(\Lambda)$, let $X^{x}$ be the  unique (mild) solution to equation \eqref{eq:Burgersmain1}. The sequence of processes $X^{n,x_n}$ converges to $X^{x}$ in $C^0([0,T];H)$, $\P$-a.s.. 
\end{theorem} 
\begin{proof} In order to prove the convergence of $X^{n,x_n}$ to $X^x$ we use a compactness method based on  a-priori estimates.
 
We need to use the following estimates: there exists $\delta_0>0$ small enough (possibly depending on $\gamma$ but independent of $n$) such that for any $\omega$, $\P$-a.s., there exist $C_T = C_T(\omega, \| x \|)>0$ and $C'_T =C'_T(\omega, \delta_0)$:
\begin{gather} \label{ddt}
\|Y^{n,x_n}\|_{L^{\infty}_t L^{2}_{\xi}} + \|Y^{n,x_n}\|_{L^{2}_tH^1_{\xi}} \le C_T.
\\ \label{ddt1} 
 \|z^n\|_{L^4_{t}W^{\delta_0,4}_{\xi}} = \|(-A)^{\delta_0/2} z^n\|_{L^4_{t}L^{2}_{\xi}} \le C_T'.
\end{gather}
Estimate \eqref{ddt} are quite standard and  can be obtained arguing as in the proof of Theorem 14.2.4 in \cite{DPZ96}. Estimate \eqref{ddt1} can be obtained following the proof of Proposition 2.1 in \cite{DPD}. Note that \eqref{ddt1}
 is a special case of \eqref{finite} when $\alpha =1$ and $0<\delta_0 < 1/4 -\gamma$.

 We will obtain that for 	$\rho>0$ small enough, there is a random variable $E_{x,\rho}:\Omega \rightarrow \R_+$ such that, $\P$-a.s.,
	\begin{equation}\label{eq:inequalitysolconvergence0}
		\|Y^{n,x_n}-e^{A t}\pi_n(x)\|_{C^{\rho}_tH^{\rho}_{\xi}} \leqslant E_{x,\rho}.
	\end{equation} 
	Indeed we have that 
	\begin{equation}\label{eq:inequalitysolconvergence1}
		\|Y^{n,x_n}-e^{A t}\pi_n(x)\|_{C^{\rho}_tH^{\rho}_{\xi}}=\frac{1}{2}\left\|\int_0^{\cdot}e^{A(\cdot -s)}\partial_{\xi}(X^{n,x_n}_s)^2\mathd s\right\|_{C^{\rho}_tH^{\rho}_{\xi}}. 
	\end{equation} 
 In order to bound the right-hand side of \eqref{eq:inequalitysolconvergence1}, we first note that, if we choose $\theta_1 \in (0,1)$ such that $\frac{1}{2} < \theta_1 <\frac{3}{4}-4 \rho$, and thus 
	\[\frac{1-\theta_1}{2} < \frac{1}{4}, \quad 1-\theta_1-\frac{1}{2}> 4 \rho -\frac{1}{4} \] 
	by the Sobolev embedding (see Remark \ref{remark:sobolevebmedding} and Remark \ref{remark:sobolevebmedding2}) and the inclusion of $L^p$-spaces on compact intervals, we have that, for any $f\in L^{\frac{2}{1-\theta_1}}_tH^{1-\theta_1}_{\xi}$, 
	\[\|f\|_{L^{4}_tW^{4\rho,4}_{\xi}} \lesssim \|f\|_{L^{\frac{2}{1-\theta_1}}_tH^{1-\theta_1}_{\xi}}.\] 
	Furthermore, by the interpolation inequality of Remark \ref{remark:interpolation1}, for any $f \in L^{\infty}_tL^2_{\xi} \cap L^2_tH^1_{\xi}$, we obtain, $\P$-a.s., 
	\[\|f\|_{L^{\frac{2}{1-\theta_1}}_tH^{1-\theta_1}_{\xi}} \lesssim \|f\|_{L^{\infty}_tL^2_{\xi}}^{\theta_1}\|f\|_{L^2_tH^1_{\xi}}^{1-\theta_1}.\]
	Using the previous observations, we obtain 
	\begin{multline}\label{eq:inequalitysolconvergence2}
		\|X^{n,x_n} \|_{L^4_{t}W^{4\rho,4}_{\xi}}^2  \lesssim\|z^n\|_{L^4_{t}W^{4\rho,4}_{\xi}}^2 + \|Y^{n,x_n} \|_{L^4_{t}W^{4\rho,4}_{\xi}}^2\\
		\lesssim \|z^n\|_{L^4_{t}W^{4\rho,4}_{\xi}}^2 + \|Y^{n,x_n}\|_{L^{\infty}_tL^{2}_{\xi}}^{2 \theta_1}\|Y^{n,x_n}\|_{L^{2}_tH^1_{\xi}}^{2(1-\theta_1)} \le C''_T,
	\end{multline}
if $0< 4 \rho \le \delta_0$, 	using estimates \eqref{ddt} and \eqref{ddt1} ($C''_T $ depends on $\omega$, $\| x\|$ and $\delta_0$).  
	 
By the regularization properties of the derivatives of heat kernel and the Sobolev inequality (see Remark \ref{remark:heatregularization} and Remarks \ref{remark:sobolevebmedding}, and \ref{remark:sobolevebmedding2} respectively) we obtain 
	\begin{multline}\label{eq:inequalitysolconvergence3}
		\left\|\int_0^{\cdot}e^{A(\cdot -s)}\partial_{\xi}(X^{n,x_n}_s)^2\mathd s\right\|_{C^{\rho}_tH^{\rho}_{\xi}} \lesssim \left\|\int_0^{\cdot}e^{A(\cdot -s)}\partial_{\xi}(X^{n,x_n}_s)^2\mathd s\right\|_{H^{\frac{1}{2}+\rho}_tH^{\rho}_{\xi}} \\
		\lesssim \|\partial_{\xi}(X^{n,x_n})^2 \|_{H^{\frac{1}{2}+\rho-m_1}_t H^{\rho-2m_2}_{\xi}}  \lesssim  \|(X^{n,x_n})^2 \|_{L^2_tH^{1+\rho-2m_2}_{\xi}}
	\\	 \lesssim \|(X^{n,x_n})^2 \|_{L^2_tH^{4\rho}_{\xi}}
		 \lesssim  \|X^{n,x_n} \|_{L^4_tW^{4\rho,4}_{\xi}}^2,
	\end{multline}
	where $m_1=\frac{1}{2}+\rho$, $\frac{3}{8} - \frac{3}{2} \rho < m_2< \frac{3}{8} +\frac{5}{2}\rho$. Note that  
	\[m_1+m_2<1 
	\]
	when $\rho<\frac{1}{36} \wedge \frac{\delta_0}{4}$ (see \eqref{ddt1}).
	In the last step we also use  that, for any $f,g\in W^{4\rho,4}_{\xi}$ then $f\cdot g \in H^{4\rho}_{\xi}$ and we have 
	\begin{equation}\label{eq:algebra}
		\|f \cdot g \|_{H^{4\rho}_{\xi}} \lesssim  \|f \|_{W^{4\rho,4}_{\xi}} \|g\|_{W^{4\rho,4}_{\xi}}.
	\end{equation}
	We refer to Theorem 5.1 of \cite{BH2021} for a proof of inequality \eqref{eq:algebra} (in the notation of \cite{BH2021} we have: $s_1 = s_2 = s =4 \rho; p_1 = p_2 =4; p =2$).
	
	If we insert in inequality \eqref{eq:inequalitysolconvergence1} the relation \eqref{eq:inequalitysolconvergence3}, by inequality \eqref{eq:inequalitysolconvergence2} 
	we obtain \eqref{eq:inequalitysolconvergence0}. Since the space $C^{\rho}([0,T];H^{\rho}(\Lambda))$ is compactly embedded in $C^0([0,T];H)$ and since $e^{At}\pi_n(x)$ and $z^n_{}$ converge to $e^{At}x$ and $z_{ }$ in $C^0([0,T];H)$, respectively, we obtain that  possibly passing to a subsequence (still denoted by   $X^{n,x_n}$)   $X^{n,x_n}$ converges to some process $\bar{X}$ in $C^0([0,T];H)$, $\P$-a.s.
  
    We now show that $\bar{X}$ is a (weak) solution to equation \eqref{eq:weaksolution}.

   Since $X^{n,x_n}$ is a solution to equation \eqref{eq:Xapparoximation1}, then, for every smooth function $f\in C^{\infty}(\Lambda)$  with compact support, we have  
  \begin{align}\label{eq:weakxn}
  	\langle X^{n,x_n}_t- x, \pi_n(f) \rangle=&\int_0^t\langle X^{n,x_n}_s ,\Delta \pi_n (f) \rangle \mathd s -\frac{1}{2}\int_0^t \langle  (X^{n,x_n}_s)^2, \partial_{\xi}\pi_n(f) \rangle \mathd s  \\ \notag
  	& +  \langle  (-A)^{\gamma} \pi_n (f), W_t\rangle.  
  	\end{align}   	
   Since $f$ is a smooth functions with compact support is a $H^s_{\xi}$ function for every $s \in \mathbb{R}$. This means that $\pi_n(f)$ converges to $f$ in $H^s_{\xi}$ for any  $s \in \mathbb{R}_+$,    and thus, by the Sobolev embedding of Remark \ref{remark:sobolevebmedding}, it converges (strongly) in $C^k_{\xi}$ for every $k\in \mathbb{N}$. Thus $\Delta \pi_n(f)$ converges to $\Delta f$ in $C^k_{\xi}$ too. Since $X^{n,x_n}$ converges in $C^0([0,T];H)$ to $\bar{X}$, $\P$-a.s,  we obtain easily passing to the limit as $n \to \infty $ that  $\bar{X}$ solves equation \eqref{eq:weaksolution}.     Thus,  by  Propositions \ref{weakmild} and \ref{cia1}   we obtain that $\bar{X}=X^{x}$. This fact proves that $X^{n,x_n}$ converges to $X^x$ almost surely. 
\end{proof}
 \begin{remark} Note that following the  proof of Theorem \ref{theorem:existenceconvergence} one  obtains   that $X^{x_n,n}$ converges strongly in $C^{\rho-\varepsilon}_tH^{\rho-\varepsilon}_\xi$ to $X^x$ (with  $0<\varepsilon<\rho$).
\qed  
 \end{remark}

For every  $\alpha\ge 0$ we introduce the process 
\begin{equation} \label{zz}
z_{\alpha}(t)=\int_0^t e^{(A-\alpha)(t-s)}(-A)^{\gamma}\mathd W_s.
\end{equation}
 which solves $\mathd z_{\alpha}(t)=(A-\alpha) z_{\alpha}(t)\mathd t+ (-A)^{\gamma}\mathd W_t,\quad z_{\alpha}(0)=0$. 

 Here we do not stress the fact that $z_{\alpha}$ does depend also on the parameter $\gamma \in \left[0,\frac{1}{4}\right)$ in order to not introduced a too heavy notation. Note that if $\kappa \ge 0$ verifies $\kappa + \gamma <1/4$ then we have 
\begin{equation}\label{s33}
(-A)^{\kappa}z_{\alpha}(t) = \int_0^t e^{(A-\alpha)(t-s)}(-A)^{\gamma + \kappa}\mathd W_s.
\end{equation}
We have two results  on the process $z_{\alpha}$. The first result is about continuous dependence of $z_{\alpha}$ with respect to its variables. The proof is given in Appendix \ref{app1}.

\begin{proposition}\label{continuo}
 Consider $\gamma \in \left[0,\frac{1}{4}\right)$. Let $\kappa \ge 0$ such that  $\kappa + \gamma <1/4$. Then   there exists a version of the random field $z_{\alpha}(t,\xi)$ (in the variables $\alpha \geqslant 0$, $t\in[0,T]$ and $\xi\in\Lambda$) which is denoted by $ \bar z_{\alpha}(t,\xi)$ such that $(-A)^{\kappa} \bar z_{\alpha}(t,\xi)$ is continuous in all its variables. 
\end{proposition}
Note that, for any $n \ge 1$, 
\begin{equation}\label{pinn}
\pi_n (-A)^{\kappa} \bar z_{\alpha}(t,\cdot)(\xi) = (-A)^{\kappa} \pi_n \bar z_{\alpha}(t,\xi) \; \text{ is continuous in all its variables,} 
\end{equation}  
for any $n \ge 1$. The second  result is     
similar to 
   Proposition 2.1 in \cite{DPD} and concerns $\alpha \ge 1$. 
\begin{proposition} \label{prop:K} 
Let $\alpha \ge 1$. Consider $\gamma \in \left[0,\frac{1}{4}\right)$,  $1 <  p <+\infty$, $\delta \in (0,1)$, $\varepsilon>0$ and $\kappa \geqslant 0$ such that $\varepsilon + \kappa < \frac{1}{4}-\gamma$.
Then there exist  non-negative random variables $K_{\eps, \delta, p}^{ ( \kappa)}$
 which is measurable with respect to the filtration generated  by $(W_t)_{t\ge 0}$, which have all bounded moments and
such that for all $\alpha \geq 1$ and $t \ge 0$, we have
$$
\| (-A)^{\kappa} z_{\alpha}(t) \|_{L^p(\Lambda)} \leq \alpha^{- 1/4 + \eps + \gamma+ \kappa} (1 + t^{\delta}) K_{\eps, \delta, p}^{(\kappa)},\;\;\; \P{-}a.s..
$$  
 \end{proposition}
 When $\kappa =0$ we set $ K_{\eps, \delta, p}^{(\kappa)} =  K_{\eps, \delta, p}.$ We prove the result in Section \ref{prop}.  Note that $ K_{\eps, \delta, p}^{(\kappa)}$ depends also on $\gamma$ (we do not indicate explicitly such dependence). 
  
 Combining the two previous results we get  the next result where we use the random variables $K_{\eps,  p}^{(\kappa)} =  K_{\eps, 1/2, p}^{(\kappa)}$ and set $c_T = (1+T^{1/2})$. 
\begin{theorem}\label{theorem:zalphaestimates} 
	Let $\alpha \ge 1$. Consider $\gamma \in \left[0,\frac{1}{4}\right)$,  $1 < p <+\infty$,
	$\varepsilon>0$ and $\kappa \geqslant 0$ such that $\varepsilon + \kappa < \frac{1}{4}-\gamma$. Let $T>0$. Then  the random field $\bar z_{\alpha}(t,\xi)$ given in Proposition \ref{continuo} 	takes values in $C^0([0,T];
	W^{2\kappa,p}(\Lambda) \cap C(\bar \Lambda)) $ 
	and it is continuous with respect to the parameter $\alpha \geqslant 1$.\\ 
	Finally, there are some random variables
	$K_{\varepsilon, p}^{(\kappa)}:\Omega \rightarrow \R_+,$ which have all bounded moments, such that,  $\P$-a.s., 
	\begin{equation}\label{eq:zalphaestimates}
		\|(-A)^{\kappa} \bar z_{\alpha}(t)\|_{L^p(\Lambda)} \leqslant \alpha^{-\frac{1}{4}+\varepsilon+\gamma+\kappa} c_T K_{\varepsilon, p}^{(\kappa)},  
	\end{equation}
	$\text{for any} \; \alpha \ge 1, \, t \in [0,T]$. 
\end{theorem} 
\begin{remark} \label{see}
Since  the random variable $\| (-A)^{\kappa} \bar  z_{\alpha}(t) \|_{L^p(\Lambda)}$ is continuous in $(\alpha,t)$ the previous estimate holds when $\omega$ is varying on an almost-sure event which does not depend on $\alpha$ and $t$. \qed
\end{remark}

Combining  the previous result with \eqref{interpo} we get (possibly multiplying the random variable $K_{\varepsilon, p}^{(\kappa)}$ by a constant $C_p$)
\begin{equation}\label{finite}
 \|(-A)^{\kappa} \pi_n \bar z_{\alpha}(t)\|_{L^p(\Lambda)} \leqslant \alpha^{-\frac{1}{4}+\varepsilon+\gamma+\kappa} c_T K_{\varepsilon, p}^{(\kappa)}, \;\; n \ge 1.
\end{equation}
Hereafter we also write 
\begin{equation} \label{yya}
Y^{x,\alpha}_t=X^{x}_t-{\bar z_{\alpha}}(t),
\end{equation}
which solves the (random) PDE
\begin{equation}\label{eq:Ymain1}
	\frac{d}{dt} Y^{x,\alpha}_t=AY^{x,\alpha}_t+\frac{1}{2} \partial_{\xi}(Y^{x,\alpha}_t+{\bar z_{\alpha}}(t))^2+\alpha {\bar z_{\alpha}}(t),\quad Y^{x,\alpha}_0=x
\end{equation}
 (it is easy to check \eqref{eq:Ymain1} by looking to the weak formulation of the SPDEs verified by $X^{x}$ and ${\bar z_{\alpha}}$). 
 Recall   the approximating   equation  \eqref{eq:Yapproximation1}, which is   a finite dimensional stochastic differential equation (SDE) defined on the space $H_n$ (this  is the $n$ dimensional subspace of $H$ generated by the first $n$ eigenfunctions of $A$).
  We write ${\bar z^n_{\alpha}}:=\pi_n({\bar z_{\alpha}})$ and 
$$Y^{n,x_n,\alpha}:=X^{n,x_n}-{\bar  z^n_{\alpha}}. 
$$
 Similarly to \eqref{eq:Yapproximation1},
the process $Y^{n,x_n,\alpha}$ solves the following random ODE 
\begin{equation}\label{eq:Yapproximation13}
	\frac{d}{dt}Y^{n,x_n,\alpha}=AY^{n,x_n,\alpha}_t+\frac{1}{2}\pi_n\left(\partial_{\xi}(Y^{n,x_n,\alpha}_t+ \bar z^n_{\alpha}(t))^2\right)+\alpha {\bar z^n_{\alpha}(t)},\quad Y^{n,x_n,\alpha}_0=\pi_n(x). 
\end{equation}

\subsection{Application of Bou\'e-Dupuis formula to Burgers equation}\label{section:applicationBurgers} 

We are going to consider the following finite dimensional SDE on $H_n  \simeq \R^n$
 (cf. \eqref{eq:BDSDE}).
\begin{gather}\label{eq:Burgersapporximation} 
	dX^{n,x_n,u}_t = \Big(A X_t^{n,x_n,u}+\frac{1}{2}\pi_n \left(\partial_{\xi}(X^{n,x_n,u}_t)^2 \right) +\pi_n (-A)^{\gamma}u(t) \Big)dt + \pi_n (-A)^{\gamma} dW_t,  \\
	 \nonumber X^{n,x_n,u} = x_n = \pi_n x, 
\end{gather} 
where $u:[0,T]\times  \Omega \rightarrow H_n$   belongs to $\mathcal{P}^{[0,T],H_n}$.  Arguing as for eq. \eqref{easy} it is easy to prove that for any $u \in \mathcal{P}^{[0,T],H_n}$ there exists a unique global  solution to \eqref{eq:Burgersapporximation}. 

Since $(-A)^{\gamma}$ and $\pi_n$ commutes, $u_n=\pi_n (-A)^{\gamma} u$ is again a finite dimensional $L^2$-integrable process taking values in $H_n$. 
Using the previous notation we have  
\begin{theorem} 
	\label{theorem:apriorigamma} For any $\gamma < \frac{1}{4}$, for any $T>0$, there are $\delta_1 > 0$,  $0<\varepsilon_1<\frac{1}{4}-\gamma$, $c_{T, \varepsilon_1} > 0$ and $\theta := 2(\frac{1}{4}-\varepsilon_1-\gamma)^{-1} > 0$ such that,
	$\P$-a.s., 
	for any $t\in[0,T],$ we have   
	\begin{equation}\label{eq:aprioriXn}
		 \| X^{n,x_n,u}_t \|_{L^2_{\xi}}^2 \leqslant e^{-\delta_1 t} \| x_n\|^2_{L^2_{\xi}}+c_{T, \varepsilon_1} \left( 1 + \left(K^{}_{\varepsilon_1,  4}\right)^{\theta} + \| u \|^2_{L^2_t L^2_{\xi}} \right).	
	\end{equation} 
\end{theorem} 
  \begin{proof} 
 	We divide the proof in two steps.
 	
{\it Step 1: Let $\alpha \ge 1$. We consider the random field
 		\begin{equation}\label{yyn}
 			Y_t^{n,x_n,u,\alpha} = X^{n,x_n,u}_t - {\bar z_{\alpha}}^n(t) ,
 		\end{equation}
 		which is the analogous of the process \eqref{eq:Yapproximation13}, and we prove some a-priori estimates for $Y_t^{n,x_n,u,\alpha}$.
 		}
 		
 \noindent 	The process $Y^{n,x_n,u,\alpha}_t$ takes values on $H_n$ and solves the following (finite dimensional) random ODE
 	\[
 	\frac{\mathd}{\mathd t}Y_t^{n,x_n,u,\alpha}=A Y_t^{n,x_n,u,\alpha}+\frac{1}{2}\pi_n\left(\partial_{\xi}(Y_t^{n,x_n,u,\alpha}+{\bar z^n_{\alpha}(t)})^2\right)+\alpha z^n_{\alpha}(t) +\pi_n (-A)^{\gamma}u(t),
 	\]
 	$\P$-a.s., for any a.e. $t \in [0,T]$ (with respect to the Lebesgue measure). 
 	By multiplying both sides for $Y_t^{n,x_n,u,\alpha}$ and then integrating with respect to $\xi$ on $\Lambda$, we obtain the following inequality ($\nabla Y_t^{n,x_n,u,\alpha} = \partial_{\xi}  Y_t^{n,x_n,u,\alpha})$
 	\begin{gather}  
 		\frac{1}{2}\frac{\mathd}{\mathd t} \| Y_t^{n,x_n,u,\alpha} \|_{L^2_{\xi}}^2 = - \| \nabla Y_t^{n,x_n,u,\alpha} \|_{L^2_{\xi}}^2 + \alpha \int_{\small \Lambda}  
 		Y_t^{n,x_n,u,\alpha} {\bar z_{\alpha}}^n(t) \mathd \xi+\nonumber\\ 
 		+\int_{\small \Lambda} (-A)^{\gamma}(Y_t^{n,x_n,u,\alpha}) \pi_n u(t) \mathd \xi
 		- \frac{1}{2}\int_{\small \Lambda} (Y_t^{n,x_n,u,\alpha}+ {\bar z_{\alpha}}^n(t))^2 \partial_{\xi}
 		Y_t^{n,x_n,u,\alpha} \mathd \xi .\label{eq:derivativeY}	 
 	\end{gather}   
 	We remark that  for any $\alpha \ge 1$ there exists $\Omega_{\alpha}\in {\mathcal F}$ with  
 	$\P(\Omega_{\alpha})=1$ such that for any $\omega \in \Omega_{\alpha}$, \eqref{eq:derivativeY} holds as an identity in $L^2([0,T])$ (i.e., the right-hand side and the left-hand side in \eqref{eq:derivativeY} as functions of time are both in $L^2([0,T])$).

 	Since we are in a finite dimensional setting it is easy to prove by the Kolmogorov test that  the process $Y_t^{n,x_n,u,\alpha}$ admits a version (still denoted by $Y_t^{n,x_n,u,\alpha}$) which is continuous in $(\alpha,t, \xi) \in [1, \infty) \times [0,T] \times \bar \Lambda$.  	Similarly,
 	\begin{gather} \label{discuss}
 		\nabla Y_t^{n,x_n,u,\alpha} = \partial_{\xi}  Y_t^{n,x_n,u,\alpha}
 	\end{gather}
 	admits a version  which is continuous in $(\alpha,t, \xi) \in [1, \infty) \times [0,T] \times \bar \Lambda$ (still denoted by $\partial_{\xi} Y_t^{n,x_n,u,\alpha}$). We can fix $n, x_n, u$ and $\alpha$ and consider  
 	$\frac{\mathd}{\mathd t} \| Y_t^{n,x_n,u,\alpha} \|_{L^2_{\xi}}^2 $ as an $L^2([0,T])$-random variable.

 	By identity \eqref{eq:derivativeY} it follows that the  $L^2([0,T])$-random variable
 	$\frac{\mathd}{\mathd t} \| Y_t^{n,x_n,u,\alpha} \|_{L^2_{\xi}}^2 $ has a version 
 	(still denoted by $\frac{\mathd}{\mathd t} \| Y_t^{n,x_n,u,\alpha} \|_{L^2_{\xi}}^2 )$ which depends  continuously on $\alpha $. In the sequel we will use all the previous  versions.   
 	 
 	
 	Noting that $ \int_{\small \Lambda} (Y_t^{n,x_n,u,\alpha})^2 \partial_{\xi}
 	Y_t^{n,x_n,u,\alpha} \mathd \xi  = 0$
 	we are going to bound the various terms on the right-hand side of equation \eqref{eq:derivativeY}. 
 	First we note that, for any $\varepsilon>0$
 	\[ \left| \alpha \int_{\small \Lambda} Y_t^{n,x_n,u,\alpha} {\bar z_{\alpha}}^n(t) \mathd	\xi \right| \leqslant \frac{\alpha^2}{2 \varepsilon} \| {\bar z_{\alpha}}^n(t) \|^2_{L^2_{\xi}} + \frac{\varepsilon}{2} \| Y_t^{n,x_n,u,\alpha} \|_{L^2_{\xi}}^2. \]
 	For the term involving $Y_t^{n,x_n,u,\alpha}$ and $u$, we obtain
 	\begin{align*}
 		\left| \int_{\small \Lambda} (-A)^{\gamma}(Y_t^{n,x_n,u,\alpha}) \pi_n u(t) \mathd \xi \right| &\leq \frac{1}{2 \varepsilon} \| \pi_n u(t)\|^2_{L^2_{\xi}} + \frac{\varepsilon}{2} \left\| (-A)^{\gamma} (Y_t^{n,x_n,u,\alpha}) \right\|_{L^2_{\xi}}^2 \\
 		&\leq \frac{1}{2 \varepsilon} \| u(t) \|^2_{L^2_{\xi}} + \frac{C \varepsilon}{2} \| \nabla Y_t^{n,x_n,u,\alpha} \|_{L^2_{\xi}}^2,
 	\end{align*}
 	since $\left\| (-A)^{\gamma} Y_t^{n,x_n,u,\alpha} \right\|_{L^2_{\xi}} \leq C \| \nabla Y_t^{n,x_n,u,\alpha} \|_{L^2_{\xi}}$ for some $C > 0$, because $\gamma < 1/4$.
 	
 	For the last term in \eqref{eq:derivativeY}, we get
 	\begin{multline}
 		\left| \int_{\small \Lambda}  (Y_t^{n,x_n,u,\alpha}+ {\bar z_{\alpha}}^n(t))^2 \partial_{\xi} Y_t^{n,x_n,u,\alpha} \mathd \xi \right| \\
 		= \Big| \underbrace{2 \int_{\small \Lambda} {\bar z_{\alpha}}^n(t) Y_t^{n,x_n,u,\alpha}\partial_{\xi} Y_t^{n,x_n,u,\alpha} \mathd \xi}_{=: I_1} + \underbrace{\int_{\small \Lambda} {\bar z_{\alpha}}^n(t)^2 \partial_{\xi} Y_t^{n,x_n,u,\alpha} \mathd \xi}_{=: I_2} \Big|.
 	\end{multline}
 	To handle $I_1$ we use the  H\"older inequality to get
 	\begin{gather*}
 		|I_1| \le \|  Y_t^{n,x_n,u,\alpha} \|_{L^4_{\xi}} \| \partial_{\xi} Y_t^{n,x_n,u,\alpha} \| \, \| \bar z_{\alpha}^n(t) \|_{L^4_{\xi}}.
 	\end{gather*}
 	By the Sobolev embedding theorem and the interpolatory inequality
 	$$
 	\| u \|_{\beta} \leq \| u \|_{\alpha}^{\frac{\gamma - \beta}{\gamma - \alpha}} \| u \|_{\gamma}^{\frac{\beta - \alpha}{\gamma - \alpha}}, \quad \alpha < \beta < \gamma,
 	$$
 	we continue with
 	$
 	\|   Y_t^{n,x_n,u,\alpha}\|_{L^4(\Lambda)} \leq C \|  Y_t^{n,x_n,u,\alpha} \|_{1/4} \leq C \|  Y_t^{n,x_n,u,\alpha}\|^{3/4} \|  Y_t^{n,x_n,u,\alpha}\|_{H^1_{\xi}}^{1/4},
 	$
 	for some deterministic constant $C > 0$ that may change from line to line. Moreover, by the Poincar\'e inequality we have
 	$$
 	\|   Y_t^{n,x_n,u,\alpha} \|_{H^1_{\xi}} \leq C \| \partial_{\xi}  Y_t^{n,x_n,u,\alpha} \|.
 	$$
 	Using Young's inequality 
 	$$
 	ab \leq \frac{3}{8} a^{8/3} + \frac{5}{8} b^{8/5}, \quad a, b \geq 0,
 	$$
 	we find that
 	\begin{gather*}
 		|I_1| \leq \varepsilon \| \nabla Y_t^{n,x_n,u,\alpha} \|_{L_{\xi}^2}^2 + c_{\varepsilon} \| {\bar z_{\alpha}}^n(t) \|_{L_{\xi}^4}^{8/3} \| Y_t^{n,x_n,u,\alpha} \|_{L_{\xi}^2}^2,
 	\end{gather*}
 	for any $\varepsilon > 0$ and some suitable $c_{\varepsilon} > 0$. Arguing similarly we find, $\P$-a.s., 
 	\begin{gather*}
 		|I_2| \leq \frac{1}{2 \varepsilon} \|{\bar z_{\alpha}}^n(t) \|_{L^4_{\xi}}^4 + \frac{\varepsilon}{2} \| \nabla Y_t^{n,x_n,u,\alpha} \|_{L^2_{\xi}}^2, \;\;\; t \in [0,T].
 	\end{gather*}
 	Putting all together for any $\varepsilon > 0$ small enough, there is $\delta > 0$ such that we have, $\P$-a.s., for any $t \in [0,T]$ a.e., 
 	\begin{multline} \label{vs}
 		\frac{1}{2} \frac{\mathd}{\mathd t} \|Y_t^{n,x_n,u,\alpha} \|_{L^2_{\xi}}^2 +
 		\delta \| \nabla Y_t^{n,x_n,u,\alpha} \|_{L^2_{\xi}}^2\\
 		\leqslant c_{\varepsilon} \|\bar  z_{\alpha}^n(t) \|_{L^4_{\xi}}^{8/3} \| Y_t^{n,x_n,u,\alpha} \|_{L^2_{\xi}}^2 + \frac{\varepsilon}{2} \| Y_t^{n,x_n,u,\alpha} \|_{L^2_{\xi}}^2 + c'_{\varepsilon} ( \|\bar  z_{\alpha}^n(t)
 		\|_{L^4_{\xi}}^4 + \alpha^2 \| \bar z_{\alpha}^n(t) \|^2_{L^2_{\xi}} + \|
 		u(t) \|^2_{L^2_{\xi}}),
 	\end{multline}
 	for suitable $c_{\varepsilon}, c'_{\varepsilon} > 0$.
 	  
 	\vskip 1mm 
 	{\it Step 2: Now we perform  the idea used in the proof in Proposition 2.2 in \cite{DPD}; this  requires to consider a random dependence of the parameter $\alpha $ by $\omega$, i.e., to consider $\alpha (\omega)$.}
 	To completely justify such method we argue as follows, taking also into account the discussion around \eqref{discuss}. 	
 	
 	First 	 
 	note that the  right-hand side and the left-hand side in \eqref{vs} are both functions in $L^2([0,T])$ as far as the dependence on time is concerned.  Hence
  the  inequality \eqref{vs} is between functions in  $L^2([0,T])$ for any $\omega \in \Omega'$, $\alpha \in [1, \infty) \cap \Q$, where 
 	$$
 	\Omega' = \bigcap_{\alpha \geq 1, \alpha \in \Q} \Omega_{\alpha};
 	$$
 	$\Omega_{\alpha}\in {\mathcal F}$ with  
 	$\P(\Omega_{\alpha})=1$. 	Clearly  $\P(\Omega') = 1$. 
 	Recall that  $\bar{z}_{\alpha}$ and so $\pi_n\bar{z}_{\alpha}$ depends in a continuous way on $(\alpha, t, \xi) $ (cf. Proposition \ref{continuo}). 
 	Taking also into account the discussion around \eqref{discuss}	
 	{\it  it follows that 
 		\eqref{vs} holds for any $\omega \in \Omega'$, $\alpha \in [1, \infty)$ and for a.e. $t \in [0,T]$} (note that $t \in A_{\omega, \alpha} \subset [0,T]$ with Leb$(A_{\omega, \alpha})=T$).  
 	
 	Therefore, let $\epsilon_1$ such that $0< \epsilon_1< 1/4 - \gamma$, we can choose $\alpha \ge 1$ depending on $\omega \in \Omega'$. We choose $\alpha = \alpha (\omega)$:    
 	\begin{equation}\label{eq:alphaproof}
 		\alpha (\omega) =\frac{1}{\varepsilon'} \Big( 
 		\left({K}_{\varepsilon_1,  4}(\omega)\right)^{\left(\frac{1}{4}-\varepsilon_1-\gamma\right)^{-1}} + 1 \Big), 
 	\end{equation}
 	for some $0 < \varepsilon' \leq 1$ (that is to be chosen later), then, for any $t\in[0,T]$, 
 	\[ \| {\bar z_{\alpha}}^n(t, \omega) \|_{L^2_{\xi}} \leq \| {\bar z_{\alpha}}^n(t, \omega) \|_{L^4_{\xi}} \leqslant \tilde{c}_T
 	(\varepsilon')^{\theta_{\gamma}},\]
 	where $\tilde{c}_T > 0$ is a suitable constant, and $\theta_{\gamma}<\frac{1}{4}-\varepsilon_1-\gamma$ is some constant depending on $\gamma<\frac{1}{4}$ (and converging to $0$ when $\gamma \rightarrow \frac{1}{4}$); here we also  use \eqref{interpo}.  

 	We obtain for $\omega \in \Omega'$, for  a.e. $t \in [0,T]$,
 	\begin{multline}
 		\frac{1}{2}\frac{d}{dt}\|{Y}_t^{n,x_n,u,\alpha(\omega)}(\omega)\|_{L^2_{\xi}}^2 +\delta \|\nabla{Y}_t^{n,x_n,u,\alpha({\omega})}(\omega) \|^2_{L^2_{\xi}}\\
 		\leqslant \left(c_{\varepsilon} {\bar c}_T (\varepsilon')^{8 \theta_{\gamma}/3} + \frac{\varepsilon}{2} \right) \|{Y}_t^{n,x_n,u,\alpha({\omega})}(\omega) \|_{L^2_{\xi}}^2 + c_{T,\varepsilon,\varepsilon'} \left( \left({K}_{\varepsilon_1,  4}(\omega)\right)^{\theta} + \| u(t, \omega) \|_{L^2_{\xi}}^2 + 1 \right), \label{eq:inequalityYderivative2} 
 	\end{multline}
 	for some suitable constants $c_{T,\varepsilon,\varepsilon'}, {\bar c_T} > 0$ and $\theta := 2(\frac{1}{4}-\varepsilon_1-\gamma)^{-1}$.

 	By the Poincare inequality, there exists  $C_P>0$ such that $C_P \|f \|_{L^2_{\xi}} \leqslant \|\nabla f \|_{L^2_{\xi}}$. We obtain an estimate like \eqref{eq:inequalityYderivative2} with the term $\delta \|\nabla{Y}_t^{n,x_n,u,\alpha({\omega})}(\omega) \|^2_{L^2_{\xi}}$ replaced by $\delta C_P \|{Y}_t^{n,x_n,u,\alpha({\omega})}(\omega) \|^2_{L^2_{\xi}}$.
 	Thus if we choose $\varepsilon'$ such that $0 < \varepsilon' \leq 1$ and 
 	$$
 	C_P \delta - c_{\varepsilon} \bar{c}_T (\varepsilon')^{8 \theta_{\gamma}/3} - \frac{\varepsilon}{2} =: \delta' > 0,
 	$$ 
 	then by the Gronwall inequality we get on $\Omega'$ (we drop the dependence on $\omega$)
 	\begin{equation}\label{eq:lastY}
 		\| {Y}_t^{n,x_n,u,\alpha} \|_{L^2_{\xi}}^2 \leqslant e^{-2\delta't}\|x_n\|^2_{L^2_{\xi}} + 2 c_{T,\varepsilon,\varepsilon'} \| u \|_{L^2_tL^2_{\xi}}^2 + \frac{c_{T,\varepsilon,\varepsilon'}(1-e^{-2\delta't})}{\delta'} \left( \left({K}_{\varepsilon_1, 4}\right)^{\theta} + 1 \right),
 	\end{equation}
for any $t \in [0,T]$. 	From inequality \eqref{eq:lastY} we can get the claim with $\delta_1 = 2 \delta'$. Indeed we have on $\Omega'$ with $\alpha = \alpha (\omega)$
 	\[\|X_t^{n,x_n,u}\|_{L^2_{\xi}}^2 \leqslant \| {Y}_t^{n,x_n,u,\alpha} \|_{L^2_{\xi}}^2 +\|{\bar z_{\alpha}}^n(t) \|_{L^2_{\xi}}^2 \leqslant \| {Y}_t^{n,x_n,u,\alpha} \|_{L^2_{\xi}}^2 + (\tilde{c}_T)^{2}(\varepsilon')^{2\theta_{\gamma}},\]
$t \in [0,T]$. 	If we insert the relation \eqref{eq:lastY} in the last inequality we get the claim.
 \end{proof}

As a consequence of the proof of Theorem \ref{theorem:apriorigamma} we can  obtain the following Corollary.
 
\begin{corollary}\label{corollary:Yalpha}
With the same hypotheses and notation of Theorem \ref{theorem:apriorigamma}, suppose that we consider the random variable
\begin{equation}\label{eq:alpha0}
	\alpha_0(\omega)=\frac{1}{\varepsilon'}\left( c_T\left({K}_{\varepsilon_1, 4}(\omega)\right)^{\left(\frac{1}{4}-\varepsilon_1-\gamma\right)^{-1}}  + 1+ F(\omega) \right),
\end{equation}
	where $0 < \varepsilon' \leq 1$ is small enough such that $C_P \delta - c_{\varepsilon} \bar{c}_T (\varepsilon')^{8 \theta_{\gamma}/3} - \frac{\varepsilon}{2} =: \delta' > 0$ (the ``upper bound'' on $\varepsilon'$ does depend only on $\gamma$ and $\varepsilon_1$) and $F:\Omega \rightarrow \R_+$ is a square integrable random variable.	
	
	Then there are some constants $c_{1,\varepsilon',T},c_{2,\varepsilon',T}>0$ (depending on $\varepsilon'$ and $T$ but not on $F$) and a constant $c_{\gamma, \varepsilon_1} > 0$ (depending on $\gamma$ and $\varepsilon_1$ but not on $F$) such that, for any $\omega$ $\P$-a.s.,
\begin{equation}
\| Y_t^{n,x_n,u,\alpha_0(\omega)} (\omega)\|_{L^2_{\xi}}^2 \leqslant e^{-\delta_1 t} \|x_n\|_{L^2_{\xi}}^2+ c_{1,\varepsilon',T}\Big (1+F^{2}(\omega)+({K}_{\varepsilon_1,  4}(\omega))^{\theta}+ \| u(\omega)\|_{L^2_tL^2_{\xi}}^2\Big),\label{eq:inequalityY1}
\end{equation}
\vskip -6mm 
\begin{multline}  
\int_0^T\|\nabla Y_s^{n,x_n,u,\alpha_0(\omega)} (\omega)\|_{L^2_{\xi}}^2\mathd s \\
\leqslant	c_{\gamma, \varepsilon_1} \|x_n\|_{L^2_{\xi}}^2 + c_{2,\varepsilon',T} \Big(1+  F^{2}(\omega)+ (	{K}_{\varepsilon_1,  4}(\omega))^{\theta}+ \| u(\omega)\|_{L^2_tL^2_{\xi}}^2 \Big). \label{eq:inequalityY2}
\end{multline}
\end{corollary}
\begin{proof}
The inequality \eqref{eq:inequalityY1} holds for $t \in [0,T]$ and can be obtained directly from \eqref{eq:inequalityYderivative2} by replacing in $\alpha$ the expression \eqref{eq:alpha0} instead of the equality \eqref{eq:alphaproof}.\\
Inequality \eqref{eq:inequalityY2} follows by \eqref{eq:inequalityYderivative2}  integrating in time from $0$ to $T$.
\end{proof} 

\subsection{Exponential estimates on $X_t^x$ and $Y^{x,\alpha}_t$}\label{section:exponentialbound}

Thanks to the a-priori estimates of Section \ref{section:applicationBurgers} and the Bou\'e-Dupuis formula of Section \ref{section:BD} we can prove 
\begin{theorem}\label{theorem:exponentialbound} 
Let us fix $T>0$.
Suppose that $X^x_t$ is solution to equation \eqref{eq:Burgersmain1}, that $Y^{n,x_n,\alpha_0}_t=X^{n,x_n}_t-z_{\alpha_0}(t)$. Using the same hypotheses and notation of Theorem \ref{theorem:apriorigamma}, and considering random variables $F$ and $\alpha_0$ of the form in the hypotheses of Corollary \ref{corollary:Yalpha}, there are some $\lambda_{\gamma},\lambda'_{\gamma},\lambda''_{\gamma}>0$ and positive constants $C_{1, \gamma, T}$, $C_{2, \gamma, T}$, $C_{3, \gamma, T, F}$, $\delta_1$, such that for every $\lambda \le \lambda_{\gamma}$, $\lambda' \le \lambda'_{\gamma}$ and $\lambda'' \le \lambda''_{\gamma}$, 
we have  
\begin{align}
	\mathbb{E}\Big[\exp\left(\lambda\sup_{t\in[0,T]}\|X_t^x \|^2_{L^2_{\xi}} \right) \Big]  \leqslant&\exp(\lambda \| x\|_{L^2_{\xi}}^2 +C_{1,\gamma,T})\label{eq:expsupX}\\
	\mathbb{E}\left[\exp\left(\lambda' \|X_t^x \|^2_{L^2_{\xi}} \right) \right]  \leqslant& \exp(\lambda'e^{-\delta_1 t} \| x\|_{L^2_{\xi}}^2+C_{2,\gamma,T})\label{eq:expXt}\\
	\mathbb{E}\left[\exp\left(\lambda''\int_0^T\|\nabla Y^{n,x_n,\alpha_0}_t\|^2_{L^2_{\xi}}\mathd t \right)\right]  \leqslant&\exp(c_{\gamma, \varepsilon_1} \lambda''\|x\|_{L^2_{\xi}}^2+C_{3,\gamma,T,F}),\;\; n \in \mathbb{N}.\label{eq:expYalpha} 
\end{align}  
\end{theorem}  
\begin{proof}
We apply Proposition \ref{proposition:BD} to a functional of the solution to \eqref{eq:Burgersapporximation}. As functional on $C^0([0,T];H_n)$ we choose $G(k)=\lambda_{\gamma}\sup_{t\in[0,T]}\|k(t)\|_{L^2_{\xi}}^2$, $k \in  C^0([0,T];H_n)$, where $\lambda_{\gamma}$ is a constant to be determined.
 We fix $t \in (0,T]$.  By Proposition \ref{proposition:BD} we have
\begin{multline}\label{eq:proofX1}
	R_n:=\log\Big(\mathbb{E}\Big[\exp\Big(\lambda_{\gamma}\sup_{t\in[0,T]}\|X^{n,x_n,0}_t \|^2_{L^2_{\xi}} \Big) \Big] \Big)\\ 
	=\sup_{u \in \mathcal{P}^{[0,T],H_n}} \mathbb{E}\Big[\lambda_{\gamma}\sup_{t\in[0,T]}\|X^{n,x_n,u}_t \|^2_{L^2_{\xi}}-\frac{1}{2}\|u\|_{L^2_tL^2_{\xi}}^2 \Big].	
\end{multline}
We now apply inequality \eqref{eq:aprioriXn} to the right-hand side of \eqref{eq:proofX1}, obtaining
\begin{multline}
	R_n \leqslant \sup_{u \in \mathcal{P}^{[0,T],H_n}}\mathbb{E}\Big[\lambda_{\gamma}\|x\|_{L^2_{\xi}}^2 + \lambda_{\gamma} c_{T, \varepsilon_1} \Big( \Big({K}_{\varepsilon_1,  4}\Big)^{\theta} + 1 \Big)\Big.\\
	\Big.+\Big(\lambda_{\gamma} c_{T, \varepsilon_1}-\frac{1}{2}\Big)\|u \|^2_{L^2_t L^2_{\xi}} \Big]. \label{eq:proofX2}
\end{multline} 
If we choose $\lambda_{\gamma}<\frac{1}{2 c_{T, \varepsilon_1}}$ the term within expectation in \eqref{eq:proofX2} is bounded from above by an expression not depending on $u$ anymore. This implies that 
\[R_n \leqslant \lambda_{\gamma}\|x\|_{L^2_\xi}^2+C_{1,\gamma,T}, \]
where $C_{1,\gamma,T}$ is given by the expression 
\[C_{1,\gamma,T}=\lambda_{\gamma} c_{T, \varepsilon_1} \mathbb{E} \left[\left({K}_{\varepsilon_1,  4}\right)^{\theta} + 1 \right],\]
	which is bounded since ${K}_{\varepsilon_1,  4}$ has all finite moments. Since, by Theorem \ref{theorem:existenceconvergence}, $X^{n,x_n,0}_t$ converges to $X^{x}_t$ in $C^0([0,T];H)$, applying the Fatou Lemma, we obtain inequality \eqref{eq:expsupX}. Inequalities \eqref{eq:expXt} and \eqref{eq:expYalpha} can be obtained in a similar way.   
\end{proof} 

\section{Further applications to Burgers equations} \label{section:applicationBurgers1}
 
\subsection{Exponential integrability of the invariant measure of Burgers equation}

In this section we prove that any invariant measure $\nu$ of equation \eqref{eq:Burgersmain1} has exponential tails (the existence and uniqueness of the invariant measure for equation \eqref{eq:Burgersmain1} is discussed in Section \ref{section:appendix}).

\begin{theorem}\label{theorem:expBurgersinvariant}
Let $\nu$ be any invariant measure to equation \eqref{eq:Burgersmain1}, then we have that
\[\int_{H} \exp(\lambda'\|x\|_{L^2_{\xi}}^2)\nu(\mathd x) < +\infty, \]
where $\lambda' \le \lambda'_{\gamma}$  ($\lambda'_{\gamma}$ is the constant in Theorem \ref{theorem:exponentialbound}).   
\end{theorem} 
\begin{proof}
By Theorem \ref{theorem:exponentialbound} and the Young inequality we
have
\[ \mathbb{E} \left[\exp (\lambda'_{}\| X_t^x \|_{L^2_{\xi}}^{2})  \right] \leqslant \frac{\exp \left( \lambda'_{} p e^{- \delta_1t}
	\| x \|^{2}_{L^2_{\xi}}\right)}{p}  + \frac{\exp (qC_{2,\gamma,T})}{q},\;\; \; t \in [0,T]  . \]
If we take the integral with respect to $\nu$ at both sides and using the
fact that $\nu$ is an invariant measure of equation \eqref{eq:Burgersmain1} it is not difficult to prove that  
\begin{multline*}
	\int_{H} \exp (\lambda'_{} \| x \|_{L^2_{\xi}}^2) \nu
	(\mathd x)= \int_{H}\mathbb{E} [\lambda'_{} \exp ( \| X_t^x \|_{L^2_{\xi}}^2) ] \nu(\mathd x)\\
	 \leqslant \frac{1}{p} \int_{H}\exp\left( \lambda'_{} p e^{- \delta_1t}\| x \|^{2}_{L^2_{\xi}}\right)\nu(\mathd x)+\frac{\exp (qC_{2,\gamma,T})}{q}
\end{multline*}
 with $1/p+ 1/q=1$.
 Now since $\delta_1 >0$  we can choose $p>1$ and $t\in[0,T]$ such that 
\[pe^{-\delta_1t}<1.\] 
With this choice we obtain 
\[\int_{H} \exp (\lambda'_{} \| x \|_{L^2_{\xi}}^2) \nu
(\mathd x) \leqslant \frac{p\exp (qC_{2,\gamma,T})}{(p-1)q}.\]
\end{proof}

	\begin{remark}\label{remark:largedeviation}
	One important application of the exponential bounds, proved in Theorem \ref{theorem:expBurgersinvariant} on the invariant measure of the Burgers equation \eqref{eq:Burgers}, could be the possibility to extend the large deviation results proved in \cite{BFZ2026} for the invariant measures of stochastic Burgers equations with trace class noise. \qed
\end{remark}

\subsection{Sharp Lipschitz improvement property of the semigroup associated with Burgers equation}

We consider the Markov semigroup $(P_t)$ associated with the Burgers equation \eqref{eq:Burgersmain1}, i.e.,  $P_t: {B}_b(H )  \rightarrow {B}_b(H )$, defined as
\[P_t\varphi(x):=\mathbb{E}\left[\varphi(X^x_t)\right],\;\;\; t \ge 0,\; x \in H.
\]
  
The next result improves locally Lipschitz estimates given in \cite{DPD} in the case of white noise. They claim  in Remark 3.5  an estimate like
\begin{equation}\label{dapdeb}
	 \|DP_t \varphi(x)\|_{H^1_{\xi}} \le \frac{ c_T}{t^{7/8}} \| \varphi \|_{L^{\infty}}(1 + \|x\|_{L^6})^4,\;\; t \in (0,T].  
\end{equation}
They explain that \eqref{dapdeb} could be justified using a suitable
approximation of the Burgers equation  - for instance by Galerkin approximations
as we have done. 

On the other hand, our estimates are like the locally Lipschitz estimates proved in \cite{DPDlincei1998} in the case of stochastic Burgers equations driven by trace class  noise. This is way we call our estimate ``sharp locally Lipschitz estimate''. 

The locally Lipschitz estimates proved in \cite{DPDlincei1998} are used   to solve  HJB equations for stochastic Burgers equations in case of trace class noise (see also \cite{DDsicon}). 
   We will treat more general   HJB equations for stochastic Burgers equations in Section  \ref{remark:HJB}.

\begin{theorem}\label{theorem:Lip}
The semigroup $P_t$ is Lipschitz improving, in the sense that for every $\varepsilon>0$ there is a constant $L_{T,\varepsilon}>0$ such that, for any $\varphi \in {B}_b(H)$, for every $t\in(0,T]$, and for every $x,x'\in H$ we have
\begin{equation}\label{eq:LipP}
	|P_t\varphi(x)-P_t\varphi(x')|\leqslant\frac{L_{T,\varepsilon}\exp\left(\varepsilon\left(\|x\|_{L^2_{\xi}}^2+\|x'\|_{L^2_{\xi}}^2\right)\right)}{\sqrt{t}}\|\varphi\|_{L^{\infty}} \|x-x'\|_{L^2_{\xi}}.
\end{equation}
\end{theorem} 
In order to prove Theorem \ref{theorem:Lip} we need some lemmas. 

First we note that the process solving the analogous of equation \eqref{eq:Burgersapporximation} is differentiable with respect the initial data $x_n$. The derivatives with respect to $x_n$ of $X^{n,x_n}:=X^{n,x_n,0}$ in the direction $h_n=\pi_n(h)\in H_n$, i.e., $\eta_t^{x_n,h_n} = D_x X^{n,x_n}_t[h_n],$ is given by the unique solution to the following (finite dimensional) differential equation (cf. Section 5.4 in  \cite{DP} which considers a trace class  noise)
\begin{equation}\label{eq:eta1}
	\frac{\mathd \eta_t^{x_n,h_n}}{\mathd t}=A\eta_t^{x_n,h_n}+2\pi_n\left(\partial_{\xi}\left(X^{n,x_n}_t \eta_t^{x_n,h_n}\right)\right),\quad \eta_{0}^{x_n,h_n}=h_n.
\end{equation}
In the next result we also use the notation $C_{\xi} = C(\bar \Lambda)$. 

\begin{lemma} \label{lemma:lemmaeta} There are some constants $L_{1,\gamma,T}>0$, $L_{2,\gamma,T}\geqslant 0$, such that,
for any $\alpha \geq 0$, $n \in {\mathbb N,}$ we have,
$\P$-a.s.,
\begin{multline}\label{eq:estimateeta}
\sup_{t \in [0,T]}\|\eta_t^{x_n,h_n}\|_{L^2_{\xi}}^2+L_{1,\gamma,T}\int_0^T \|\nabla\eta^{x_n,h_n}_s\|^2_{L^2_{\xi}} \mathd s \\
\leqslant \exp\left(L_{2,\gamma,T}\left( \|{\bar z_{\alpha}}^n
\|_{ L_{t}^{2}  C^{}_{\xi}}^{2}
+ \| Y^{n,x_n,\alpha}\|_{L^2_tH^1_{\xi}}^{4/3} \right) \right)\|h_n\|_{L^2_{\xi}}^2.
\end{multline}   
\end{lemma}  
\begin{proof}
We follow  a part of the proof of Lemma 5.12 of \cite{DP}. The main differences in the proof are due to the lack of (space) regularity of the process ${\bar z_{\alpha}}$ in our case. In Lemma \ref{lemma:approximationderivatives} we will use
the fact that \eqref{eq:estimateeta} contains the term $ \|{\bar z_{\alpha}}^n
\|_{ L_{t}^{2}  C^{}_{\xi}}^{2}$.

First we multiply both sides of \eqref{eq:eta1} by $\eta_t^{x_n,h_n}$ and then 
we integrate by parts with respect to the space variable obtaining
\begin{equation}\label{eq:eta2}
	\frac{1}{2}\frac{\mathd}{\mathd t}\|\eta_t^{x_n,h_n} \|_{L^2_{\xi}}^2+\|\nabla \eta_t^{x_n,h_n} \|_{L^2_{\xi}}^2 = 2\int_{\Lambda}{\partial_{\xi}({\bar z_{\alpha}}^n(t))|\eta_t^{x_n,h_n}|^2\mathd \xi}+2\int_{\Lambda}{\partial_{\xi}(Y^{n,x_n,\alpha}_t)|\eta_t^{x_n,h_n}|^2\mathd \xi} 
\end{equation}
 (recall that we are considering Dirichlet boundary conditions).
We are also using that 
\begin{equation*} 
	Y_t^{n,x_n,\alpha} = X^{n,x_n}_t - {\bar z_{\alpha}}^n(t).
\end{equation*}
The second term on the right-hand side of \eqref{eq:eta2} can be bound as in Lemma 5.12 of \cite{DP} obtaining 
\[\left|\int_{\Lambda}{\partial_{\xi}(Y^{n.x_n,\alpha}_t)|\eta_t^{x_n,h_n}|^2\mathd \xi} \right| \leqslant C \|Y^{n,x_n,\alpha}_t \|_{H^1_{\xi}}^{4/3} \|\eta_t^{x_n,h_n}\|_{L^2_{\xi}}^2+\frac{1}{4} \|\nabla\eta_t^{x_n,h_n}\|_{L^2_{\xi}}^2,\]
where $C>0$ is a suitable constant. 
To treat the first term on the right-hand side of \eqref{eq:eta2} we compute
\begin{gather*}   
\int_{\Lambda}\partial_{\xi}(\bar z_{\alpha}^n(t))|\eta_t^{x_n,h_n}|^2\mathd \xi = 
- 2\int_{\Lambda} \bar z_{\alpha}^n(t) \,\eta_t^{x_n,h_n} \partial_{\xi} (\eta_t^{x_n,h_n}) \mathd \xi.
\end{gather*}     
Thus we obtain 
\[\left|\int_{\Lambda}{\partial_{\xi}({\bar z_{\alpha}}^n(t))|\eta_t^{x_n,h_n}|^2\mathd \xi} \right| \leqslant \varepsilon \|\nabla\eta_t^{x_n,h_n}\|_{L^2_{\xi}}^2 + C_{\varepsilon} \|{\bar z_{\alpha}}^n(t)\|_{
C^{}_{\xi}}^{2} \|\eta_t^{x_n,h_n}\|^2_{L^2_{\xi}},
\] 
for a suitable constant $C_{\varepsilon} > 0$.  From the previous reasoning, choosing $\varepsilon>0$ small enough,  we find that there exists  a $\delta>0$ and a constant $C'' \geqslant 0$ such that
\[	
\frac{\mathd}{\mathd t}\|\eta_t^{x_n,h_h} \|_{L^2_{\xi}}^2+\delta \|\nabla \eta_t^{x_n,h_h} \|_{L^2_{\xi}}^2 \leqslant C'' \Big(\| {\bar z_{\alpha}}^n(t)\|^{2}_{C^{}_{\xi}}
+ \|Y^{n,x_n,\alpha}_t \|_{H^1_{\xi}}^{4/3} \Big)\|\eta_t^{x_n,h_h} \|_{L^2_{\xi}}^2.\]
The claim follows from an application of the Gronwall inequality.
\end{proof}

	\begin{remark}\label{remark:continuityeta} 
 Using the a-priori bounds of Lemma \ref{lemma:lemmaeta} and the method of proof of Theorem \ref{theorem:existenceconvergence}, one can prove that, $\P$-a.s.,  the processes $(\eta^{x_n,h_n}_t)$ converges in $C^{0}_tL^2_{\xi}$ to a process $(\eta^{x,h}_t)$ that is the unique (weak and mild) solution to the (random) PDE 
	\begin{equation}\label{eq:etalimit}
		\frac{\mathd \eta^{x,h}_t}{\mathd t}=A\eta^{x,h}_t+2\partial_{\xi}\left(X^{x}_t \eta_t^{x,h}\right),\quad \eta^{x,h}_0=h.
	\end{equation}
Indeed we have (cf. \eqref{eq:inequalitysolconvergence3}) 
\begin{gather*}
		\left\|\int_0^{\cdot}e^{A(\cdot -s)}\partial_{\xi}(X^{n,x_n}_s 
		\eta^{x_n,h_n}_s)\mathd s\right\|_{C^{\rho}_tH^{\rho}_{\xi}}  
		 \lesssim \|X^{n,x_n} \eta^{x_n,h_n} \|_{L^2_tH^{4\rho}_{\xi}}
	\\	 \lesssim  \|\eta^{x_n,h_n} \|_{L^2_tH^{1}_{\xi}}  \|X^{n,x_n} \|_{L^2_t W^{4\rho,4}_{\xi}}. 
\end{gather*}
Using the bounds \eqref{eq:inequalitysolconvergence2} and \eqref{eq:estimateeta} we obtain that 
$(\eta^{x_n,h_n}_t)$ is compact in  $C^{0}_t L^2_{\xi}$. Arguing as in the proof of 
Theorem \ref{theorem:existenceconvergence} we obtain that  $(\eta^{x_n,h_n}_t)$ converges in $C^{0}_tL^2_{\xi}$ to $(\eta^{x,h}_t)$. By  \eqref{eq:estimateeta}  we also get that  $\eta^{x_n,h_n}$
weakly 	converges to  $\eta^{x,h} $ in $L^2_tH^1_{\xi}$.

	
	Exploiting the lower semicontinuity of the norm with respect to weak convergence we are able to deduce, from  \eqref{eq:estimateeta}, that the process $(\eta^{x,h}_t)$
	satisfies the a-priori estimates
\begin{multline}\label{eq:estimateeta2}
	\sup_{t \in [0,T]}\|\eta_t^{x,h}\|_{L^2_{\xi}}^2+L_{1,\gamma,T}\int_0^T \|\nabla\eta^{x,h}_s\|^2_{L^2_{\xi}} \mathd s \\
	\leqslant \exp\left(L_{2,\gamma,T}\left( \|{\bar z_{\alpha}}
	\|_{ L_{t}^{2}  C^{}_{\xi}}^{2}
	+ \| Y^{x,\alpha}\|_{L^2_tH^1_{\xi}}^{4/3} \right) \right)\|h\|_{L^2_{\xi}}^2.
\end{multline} 	
We can also prove the local  Lipschitzianity of the process $X^{x}_t$ with respect to $x \in L^2_{\xi}$ (uniformly in $t \in [0,T]$). Indeed inequality \eqref{eq:estimateeta} implies that the process $X^{x_n,n}_t$ is locally  Lipschitz in  $x$ (recall that  $x_n=\pi_n(x)$) in
a uniform way with respect to both $t$  and $n\in \mathbb{N}$; to this purpose see also the bound \eqref{ddt} on  $Y^{n,x_n}$  or   \eqref{eq:estimateeta1}. By Theorem  \ref{theorem:existenceconvergence} 
we 
know that $X^{x_n,n}_t$ (as a function of $x$ taking values in $C^0_tL^2_{\xi}$) converges 
to $X^x_t$.  Furthermore, by the lower semicontinuity of the Lipschitz norm with respect to pointwise convergence, and the uniform bound \eqref{eq:estimateeta} we obtain the local Lipschitzianity
of $X^x_t$ with respect to $x$ (uniformly in $t \in [0,T]$).

Moreover, differentiating with respect to the spatial variabile in equation \eqref{eq:eta1}	(cf. Proposition 3.3 in \cite{DPDlincei1998}) and arguing as before we get the local Lipschitzianity of the process $\eta^{x,h}_t$ with respect to $x\in L^2_{\xi}$ as a random variable taking values in $L^2_tL^2_{\xi}$. \qed	
\end{remark} 

We define by $(P^n_t)$ the family of   Markov semigroups defined as
\[P^n_t\varphi(x):=\mathbb{E}\left[\varphi(X^{n, x_n }_t)\right],
\]
where $x_n = \pi_n(x)$, $\varphi \in { B}_b(H)$, $t \ge 0$.

The next result can be seen as an improvement of the estimates in Proposition 3.1 of \cite{DPD}. Indeed in such result they only consider  estimates of $\eta_t^{x_n,h_n}$ with an exponential 
weight, i.e.,
they  bound  terms like  
$\E \Big[{\exp \big(-c \int_0^t \|X^{x}_s\|_{L^4_{\xi}}^{8/3}ds\big)}  \cdot  \| \eta_t^{x,h}\|^2 \Big ]$. 

\begin{lemma} \label{deri}
For any $h_n \in H_n$, $0 < \varepsilon < \lambda''_{\gamma} c_{\gamma, \varepsilon_1}$ ($\lambda''_{\gamma} > 0$ from Theorem \ref{theorem:exponentialbound} and $c_{\gamma, \varepsilon_1} > 0$ from Corollary \ref{corollary:Yalpha}),  we have for some positive  constants $L_{1,\gamma,T}$,
 and  $C''_{\varepsilon, T}$
\begin{multline}\label{eq:estimateeta1}
\E \sup_{t \in [0,T]}\|\eta_t^{x_n,h_n}\|_{L^2_{\xi}}^2+
 L_{1,\gamma,T} 
  \E \int_0^T \|\nabla\eta^{x_n,h_n}_s\|^2_{L^2_{\xi}} \mathd s 
\leqslant \exp\left( \varepsilon \|x\|_{L^2_{\xi}}^2+ C''_{\varepsilon,T}\right)\|h_n\|_{L^2_{\xi}}^2.
\end{multline}  
\end{lemma}
 \begin{proof}
Let us fix $\epsilon_2 >0$ and $\kappa  >0 $ such that $0 < \kappa + \epsilon_2 < 1/4 - \gamma$. Recall that as in Theorem \ref{theorem:apriorigamma}  $0<\varepsilon_1<\frac{1}{4}-\gamma$. 
By the Sobolev embedding  we know that there exists $p_0=p_{\kappa, \gamma}$ large enough such that $W^{2\kappa, \, p_{0}} \subset C(\bar \Lambda)$. 

By Theorem \ref{theorem:zalphaestimates} and \eqref{finite} we get
 \[
 \|{\bar z_{\alpha}}^n(t) \|_{C_{\xi}^{}} \lesssim \|(-A)^{\kappa}{\bar z_{\alpha}}^n(t)\|_{L^{p_{0}}_{\xi}  }   \lesssim \alpha^{-\frac{1}{4}+\varepsilon_2+\gamma+\kappa} \, K_{\varepsilon_2,p_{0}}^{(\kappa)}, 
 \] 
for all $t \in [0,T]$. Recall that 
 the random variable $\| (-A)^{\kappa} \bar  z_{\alpha}^n(t) \|_{L^{p_0}_{\xi}}$ is continuous in $(\alpha,t)$ and so the previous estimate holds when $\omega$ is varying on an almost-sure event $\Omega'$ which does not depend on $\alpha$ and $t$. 
  (cf. Remark \ref{see}).
We choose, for $\omega \in \Omega'$, 
\[\alpha_0(\omega)
=\frac{1}{\varepsilon'}\Big( \left({K}_{\varepsilon_1, 4}(\omega)\right)^{\left(\frac{1}{4}-\varepsilon_1-\gamma\right)^{-1}} +\left(K_{\varepsilon_2,p_{0}}^{(\kappa)}(\omega)\right)^{\left(\frac{1}{4}-\varepsilon_2-\gamma-\kappa\right)^{-1}} + 1 \Big), \]
where $0 < \varepsilon' \leq 1$ satisfies the conditions of Corollary \ref{corollary:Yalpha}.  For every $t\in[0,T]$, we have on $\Omega'$
\[
\|{\bar z_{{\alpha}_0(\omega)}}^n(t, \omega) \|_{C_{\xi}}  \leqslant C_T,
\]  
for some constant $C_T>0$ (independent of $\omega$, $t \in [0,T]$ and $n \ge 1$). 

Now we can apply Corollary \ref{corollary:Yalpha}, by taking $F(\omega)= \left(K_{\varepsilon_2,p_{0}}^{(\kappa)}(\omega)\right)^{\left(\frac{1}{4}-\varepsilon_2-\gamma-\kappa\right)^{-1}}$ (which have all finite moments) obtaining 
\begin{multline}\label{eq:estimateexponentialderivatives}
	\mathbb{E}\Big[ \exp\Big(L_{2,\gamma,T}\left( \|\bar z_{\alpha_0}^n\|_{L^{2 }_{t}C^{}_{\xi}}^{2}+ \| Y^{n,x_n,\alpha_0}\|_{L^2_tH^1_{\xi}}^{4/3} \Big) \Big) \right]\\
\leqslant \mathbb{E}\left[\exp\left(C'_{\varepsilon,T} + \frac{\varepsilon}{c_{\gamma, \varepsilon_1}} \| Y^{n,x_n,\alpha_0}\|_{L^2_t H^1_{\xi}}^{2} \right)  \right] \leqslant \exp(\varepsilon \|x\|_{L^2_{\xi}}^2+ C''_{\varepsilon,T}),
\end{multline}
by \eqref{eq:expYalpha} for some constants $c_{\gamma, \varepsilon_1}, C'_{\varepsilon, T}, C''_{\varepsilon, T} > 0$. Using the previous estimates and 
  Lemma \ref{lemma:lemmaeta} we obtain the assertion. 
 \end{proof}

\begin{lemma}\label{lemma:approximationderivatives} Consider $\alpha \ge 1$. 
For any $h_n \in H_n$, $0 < \varepsilon < \lambda''_{\gamma} c_{\gamma, \varepsilon_1}$ ($\lambda''_{\gamma} > 0$ from Theorem \ref{theorem:exponentialbound} and $c_{\gamma, \varepsilon_1} > 0$ from Corollary \ref{corollary:Yalpha}), $K_{T, \varepsilon} > 0$ and $\varphi \in {C}_b (H)$, we have
\[|D_{h_n}(P^n_t\varphi)(x)|\leqslant \frac{K_{T,\varepsilon}\exp\left(\varepsilon\|x\|_{L^2_{\xi}}^2\right)}{\sqrt{t}}\|\varphi\|_{L^{\infty}} \|h_n\|_{L^2_{\xi}},\;\; t \in (0,T], \, x \in H,\; n \ge 1.
\] 
\end{lemma}
\begin{proof}  

 Since equation \eqref{eq:Xapparoximation1} is a finite-dimensional equation like \eqref{easy} with a non-degenerate matrix $C$ we can apply the Bismut-Elworthy-Li formula  (see Theorem 2.1 of \cite{ElworthyLi1994}) to the semigroup $P^n_t$. This is done also  in Section 4 of \cite{DPDlincei1998}.  We  obtain  
\begin{equation}\label{eq:Bismut-Elwrothy-Li}
	D_{h_n}(P^n_t\varphi)(x)=\frac{1}{t}\mathbb{E}\left[\varphi(X^{n,x_n}_t)\int_0^t {(-A)^{-\gamma}\left(\eta_s^{x_n,h_n}\right) (\pi_n\mathd W_s)} \right],\;\; t\in(0,T].
\end{equation} 
If we use in the previous equation the estimates \eqref{eq:estimateeta1}  and the H\"older inequality   we obtain
\begin{align*}  
	&|D_{h_n}(P^n_t\varphi)(x)| \lesssim \frac{1}{t}\|\varphi\|_{L^{\infty}} \left(\mathbb{E}\left[\int_0^t\|(-A)^{-\gamma}(\eta^{x_n,h_n}_s)\|_{L^2_{\xi}}^2\mathd s \right] \right)^{1/2} 	\\
	&\lesssim \frac{\|\varphi\|_{L^{\infty}}}{\sqrt{t}} \exp(\varepsilon \|x \|_{L^2_{\xi}}^2)\|h_n \|_{L^2_{\xi}}.
\end{align*}
\end{proof}

\begin{remark}\label{remark:c1Pn}
	Thanks to formula \eqref{eq:Bismut-Elwrothy-Li}, using the continuity of $\eta^{x_n,h_n}$  with respect to $x_n$, the continuity of $X^{n,x_n}$ with respect to $x_n$, the bounds \eqref{eq:estimateeta} for $\eta^{x_n,h_n}$, and inequality \eqref{eq:estimateexponentialderivatives} on the expectation of the exponential of $z_{\alpha_0}^n$ and $Y^{n,x_n,\alpha}$, it is easy to see that $D_{h_n} P^n_t\varphi$ is continuous whenever $\varphi$ is continuous. This means that $P^n_t\varphi$ belongs to $C^1(H)$ for $t \in (0,T]$. \qed
\end{remark}

	\begin{remark}\label{remark:logHarnack}
		The proof of Lemma \ref{deri} can be easily modified in order to obtain the following assertion: for any $\varphi \in C^1_b(H)$ we have that 
		\begin{equation}\label{eq:squareinsideout}
			\|DP^n_t(\varphi)(x)\|_{L^2_{\xi}}^2 \leqslant (P^n_t\|D\varphi\|_{L^2_{\xi}}^2)(x)\,  \exp(\varepsilon \|x\|^2_{L^2_{\xi}}+\Phi_{\varepsilon}(t)),
		\end{equation} 
		where $\Phi:\R_+ \rightarrow [0,+\infty)$ is some continuous increasing function. 
		Indeed we have 
		\begin{multline*} 
			|\langle DP_t^n\varphi(x), h_n \rangle|=\left|\mathbb{E}\left[\langle D\varphi(X_t^{n,x_n}), \eta^{x_n,h_n}_t \rangle  \right]\right| \leqslant \mathbb{E}\left[ \|D\varphi(x)\|_{L^2_{\xi}}\| \eta^{x_n,h_n}_t\|_{L^2_{\xi}}  \right]\\
			\leqslant  \left(\mathbb{E}\left[ \|D\varphi(X_t^{n,x_n})\|_{L^2_{\xi}}^2\right]\right)^{\frac{1}{2}} \left(\mathbb{E}\left[ \| \eta^{x_n,h_n}_t\|_{L^2_{\xi}}^2\right]\right)^{\frac{1}{2}}\\
			\leqslant \left(\mathbb{E}\left[ \|D\varphi(X_t^{n,x_n})\|_{L^2_{\xi}}^2\right]\right)^{\frac{1}{2}} \exp(\varepsilon \|x\|_{L^2_{\xi}}^{2}+\Phi(t))\|h_n\|_{L^2_{\xi}},
		\end{multline*} 
		where in the last step we use inequality \eqref{eq:estimateeta1}. By taking the sup when $\|h_n\|_{L^2_{\xi}}=1$, and observing that $\mathbb{E}\left[ \|D\varphi(X_t^{n,x_n})\|_{L^2_{\xi}} ^2\right]=P_t^n(\|D\varphi\|_{L^2_{\xi}}^{2})(x)$ we obtain inequality \eqref{eq:squareinsideout}. \\
		The reason why inequality \eqref{eq:squareinsideout} is important is that, as it has been pointed out at the beginning of Section 3 of \cite{WWLHaranck2011}, inequality \eqref{eq:squareinsideout} is the central property through which one can prove log-Harnack inequality for a Markov semigroup associated with an SPDE (see the proof of Proposition 3.3 of \cite{WWLHaranck2011}, for a detailed proof of this relation in the case of stochastic Burgers equation with trace class noise, see also the proof of Proposition 3.2 of \cite{RW2010} and the proof of Theorem 3.1 of \cite{WZ2014}). \qed
	\end{remark}

\begin{proof}[Proof of Theorem \ref{theorem:Lip}] Let $\varphi$ be a bounded and continuous function defined on $H$. Thanks Lemma \ref{lemma:approximationderivatives} and Remark \ref{remark:c1Pn} $P_t^{n}\varphi$  belongs to $C^1(H)$,  for $t \in (0,T],$ and, by the fundamental theorem of calculus, we obtain 
	\[P^n_t\varphi(x)-P^n_t \varphi(y)=\int_0^1{D_{\pi_n(x-y)}P^n_t\varphi(x+\tau (y-x))}\mathd \tau.\]
	Thus, applying Lemma \ref{lemma:approximationderivatives}, and by the convexity of the square, we get, for $x, y \in H,$
	\begin{equation}
		|P^n_t\varphi(x)-P^n_t\varphi(y)| \leqslant \frac{L_{T,\varepsilon}\exp\left(\varepsilon\|x\|_{L^2_{\xi}}^2+\varepsilon\|y\|_{L^2_{\xi}}^2\right)}{\sqrt{t}}\|\varphi\|_{L^{\infty}}\|x-y \|_{L^2_{\xi}}.
	\end{equation}
	Since $\varphi$ is bounded and continuous we can take the limit as $n \rightarrow \infty$. Indeed, since $X^{n,x_n}$ converges to $X^x$ in $C^0([0,T];H)$ (see Theorem \ref{theorem:existenceconvergence}), we have that $P^n_t\varphi$ converges to $P_t\varphi$ pointwise, and thus inequality \eqref{eq:LipP} holds for any $\varphi \in C_b(H)$.\\

	We now proceed like in Lemma 7.1.5 of \cite{DPZ96}.  If we consider the (signed) measure $\lambda_{x,y}(\mathd z)$ on $H$ defined as 
	\[\int_{H}\varphi(z) \lambda_{x,y}(\mathd z)=\mathbb{E}[\varphi(X^x_t)-\varphi(X^y_t)]=P_t\varphi(x)-P_t\varphi(y),\]
	for any bounded measurable function $\varphi$, by Hahn decomposition theorem   we have 
	\begin{multline*}
		|\lambda_{x,y}|(H)=\sup_{\varphi \in C_b(H), \;\| \varphi\|_{L^{\infty}} \le 1 } \Big|\int_{H}\varphi(z) \lambda_{x,y}(\mathd z) \Big|\\
		\leqslant\frac{L_{T,\varepsilon}\exp\big(\varepsilon\|x\|_{L^2_{\xi}}^2+\varepsilon\|y\|_{L^2_{\xi}}^2\big)}{\sqrt{t}}\|x-y \|_{L^2_{\xi}},
	\end{multline*}
	and thus, for any $\varphi \in {B}_b(H)$,
	\[|P_t\varphi(x)-P_t\varphi(y)|\leqslant\|\varphi\|_{L^{\infty}}|\lambda_{x,y}|(H) \leqslant \frac{\|\varphi\|_{L^{\infty}}L_{T,\varepsilon}\exp\big(\varepsilon\|x\|_{ L^2_{\xi}}^2+\varepsilon\|y\|_{L^2_{\xi}}^2\big)}{\sqrt{t}}\|x-y \|_{L^2_{\xi}}.\]
\end{proof}

\begin{remark}\label{remark:continuityphi} By a standard argument based on the semigroup law and Theorem \ref{theorem:Lip}, writing $P_t^n = P_{t/2}^n P_{t/2}^n$,  it follows that Lemma \ref{lemma:approximationderivatives} holds even for $\varphi \in B_b(H).$ \qed
\end{remark}


\subsection{Related results based on the    exponential estimates}	
 
In this section we  briefly discuss other possible applications of our exponential
 estimates. 
 Differently from the previous applications, we provide here 
only a brief sketch of the proofs.
	
\subsubsection{Solutions to infinite dimensional Hamilton-Jacobi-Bellman equation associated with the Burgers SPDE}\label{remark:HJB} 

First we note that the Lipschitz improvement property of Theorem \ref{theorem:Lip} can be generalized by proving a bound on the derivative $DP_t \varphi$.
	
\begin{remark}\label{remark:C1DP}
 Let $\varphi \in C_b(H)$.	Using Remark \ref{remark:continuityeta} (together with Remark \ref{remark:continuityphi}) one can prove that 
 for any $t \in (0,T]$, $x,h \in H$, there exists the directional derivative  $D_{h}(P_t\varphi)(x)$ of $P_t\varphi$ in $x$ along the direction $h$. Moreover, 
the Bismut-Elworthy-Li formula \eqref{eq:Bismut-Elwrothy-Li} pass to the limit as $n \rightarrow +\infty$ and thus we have 
 \begin{equation}\label{eq:Bismut-Elwrothy-Li2}
 	D_{h}(P_t\varphi)(x)=\frac{1}{t}\mathbb{E}\left[\varphi(X^{x}_t)\int_0^t {(-A)^{-\gamma}\left(\eta_s^{x,h}\right) (\mathd W_s)} \right],\;\; t\in(0,T].
 \end{equation}
 A similar idea is also used  in Proposition 4.1 of \cite{DPDlincei1998}. 
 Exploiting Remark \ref{remark:continuityphi}, and by the  local Lipschitzianity of $\eta^{x,h}$ with respect to $x\in L^2_{\xi}$ discussed in Remark \ref{remark:continuityeta}, it is not difficult to prove that $D_{h}P_t\varphi(x)$ is continuous with respect to $x\in L^2_{\xi}$ (uniformly in $h \in H$ with $\|h\|\le 1$), and thus $P_t\varphi \in C^1(H)$. Finally inequality \eqref{eq:LipP} implies
 \begin{equation}\label{eq:boundDP} 
 	 \|DP_t\varphi (x)\|_{L^2_\xi} \leqslant {L_{T,\varepsilon}\exp\left(2\varepsilon\|x\|_{L^2_{\xi}}^2\right)}\, {{t}^{-1/2}}\, \|\varphi\|_{L^{\infty}},\;\;  x \in H, \, t \in (0,T]. \qed
 \end{equation}
\end{remark}

Thanks to Remark \ref{remark:C1DP} and the bound \eqref{eq:boundDP} we can generalize the result about the Hamilton-Jacobi-Bellman equation associated with the stochastic Burgers equation proved in \cite{DDsicon,DPDlincei1998} in the case of trace class noise. We can treat  
\begin{align*}
	\begin{aligned}
		&	\begin{aligned}
			\partial_t u(t,x) = &g(x) +
			\langle Ax+B(x), Du(t,x)\rangle -F_0(x, Du(t,x)) 
			+ \frac{1}{2} Tr(Q D^2 u(t,x)),
		\end{aligned} 
		\\&
		u(0,x) = \varphi(x),\quad  x \in H, \quad  t \in ]0,T],
	\end{aligned} 
\end{align*}
where $\varphi, g \in C_b(H)$, $B(x) = \frac{1}{2} \partial_{\xi} (x^2)$, $x = H = L^2(\Lambda)$.  In \cite{DPDlincei1998} they consider a trace class operator $Q$. 
Here we can treat  $Q = (-A)^{2 \gamma}$, $\gamma \in  [0,1/4)$. 

One assumes that $F_0: H \times H \to \R$ is continuous and bounded and verifies     
\[|F_0(x, v_1)-F_0(x, v_2)|\leqslant C e^{-\delta \|x \|_{L^2_\xi}^2} \;\; \|v_1-v_2 \|_{L^2_{\xi}}, \;\; x, v_1, v_2 \in H,
\]
for some $C,\delta>0$. We interpret  the infinite dimensional parabolic PDE in the mild formulation 
\begin{equation}\label{eq:HJB}
	u(t, \cdot)=P_t \varphi+\int_0^tP_{t-s}[F_0(\cdot ,Du(s, \cdot)) + g]\mathd s. 
\end{equation}
Equation \eqref{eq:HJB} can be solved by a fixed point argument in the Banach space $Z_{T, \delta}$ consisting of all continuous and bounded functions $u : [0,T] \times H \to \R$ such that for all $t \in (0,T]$, we have $u(t, \cdot ) \in C^1(H)$ and moreover 
 \[
 \|u\|_{Z_{T,\delta}}:=\sup_{t\in [0,T]}\|u(t, \cdot)\|_{C_b(H)}+ \sup_{t\in(0,T]}\left(t^{1/2}\sup_{x\in H}\left(e^{-\delta \|x\|^2}\|Du(t,x)\|_{L^2_{\xi}}\right) \right)< \infty.
 \] 
Following the proof of Theorem 5.1 in \cite{DPDlincei1998} and using inequality \eqref{eq:boundDP} of Remark \ref{remark:C1DP}, it is possible to prove that \eqref{eq:HJB} admits a unique solution. 

This result can be applied, using an usual argument of Dynamic Programming Principle to study 
an optimal control problem as in  formulae (1.5)-(1.8) in \cite{DPDlincei1998}   (see also \cite{DDsicon}).

\subsubsection{An application to an SPDE with a singular drift}\label{girsanov}

   We introduce  the following  SPDE which includes \eqref{eq:Burgers}  as a special case when $\gamma =0$ (the white 
noise case)
  \begin{equation}
  \label{cheef}
X_t = e^{t A }x +\int_{0}^{t} e^{(  t-s)  A}B_0  (
X_{s})ds
 + \frac{1}{2}\int_{0}^{t} e^{(  t-s)  A}  \partial_{\xi} (
X_{s}^2)ds 
+  
\int_{0}^{t}e^{(  t-s)  A}d W_{s}, 
 \end{equation}
$x \in H = L^2(\Lambda)$, $t \in [0,T]$ (we write $X_t = X_t^{x}$). For the sake of simplicity here we only consider the case $\gamma =0$. We assume that  
 $B_0: H \to H $ is Borel and verifies   $\|B_0(x) \| \le c_0 +  \delta \, \|x \|,$ $x \in H, $ for some $c_0 \ge 0$, and $\delta >0$ small enough such that $\delta^2 T \le \lambda_{\gamma}$  (the constant $\lambda_{\gamma} >0$ is given in Theorem \ref{theorem:exponentialbound}).  
 
 One can prove weak existence for \eqref{cheef}, for  any $x \in H$, using our exponential estimates. 
  Indeed one can use the Girsanov theorem (cf. the proof of   part (i) in Proposition 25 in \cite{P21} and see 
  Appendix in \cite{DFPR13}).  

Let us fix $x \in H.$ Let $V = (V_t)$ be the unique  mild solution to the Burgers equation \eqref{cheef}
 when $B_0=0$ such that  $V_0 =x$. This is defined on  a  filtered probability space 
 $( 
\Omega, {\mathcal F},$ $  
 ({\mathcal F}_{t}), \P )$  
  on which it is defined the
cylindrical Wiener process $W$ on $H$.
 Set  
\begin{gather*} 
b(s) = B_0(V_s),\;\; s \in [0,T],
\end{gather*}
and note that $\|b(s)\| \le c_0  +   \delta \, \|V_s\|$, $s \in [0,T]$.  The process   $(b(s))_{s \in [0,T]}$ is  progressively measurable  and clearly verifies $\E \int_0^T
\|b(s)\|^2    ds < \infty $. Moreover, we know by Theorem \ref{theorem:exponentialbound}   that 
\begin{gather} \label{novikov}
\E  \big [ e^{\frac{1}{2} \int_0^T  \|b(s)  \|^2 ds}\big]  \le \, C_T \, 
 \E  \big [ e^{\delta^2  \int_0^T  \| V_s \|^2 ds}\big] < \infty.
\end{gather}  
Let ${U_t = \sum_{k \ge 1} \int_0^t \langle b(s) ,  e_k\rangle_{H}  \, dW^{(k)}_s}$, $t  \in [0,T]$. By Proposition 17 in \cite{DFPR13} we know that $\tilde W^{(k)}_t = W^{(k)}_t - \int_0^t \langle  e_k, b(s)\rangle_{H} ds$, $t \in [0,T]$, $k \ge 1,$ are independent real Wiener processes on $(
\Omega, {\mathcal F},$ $
 ({\mathcal F}_{t}), \tilde \P )$.   
  Here 
  the probability measure   
 $
 {\tilde \P = e^{ U_T \,   
 - \, \frac{1}{2} \int_0^T  \|b(s)  \|_{ }^2 ds} \, \cdot  \P}
 $ is equivalent to $\P$ (the quadratic variation process $\langle U\rangle_t$ $= \int_0^t \big \|b(s) \big \|^2 ds$, $t \in [0,T]$). 
  
 Hence  $\tilde W_t = \sum_{k \ge 1}  W^{(k)}_t  e_k$, $t \in [0,T]$, is a cylindrical Wiener process on $H$ defined on $(
\Omega, {\mathcal F},$ $  
 ({\mathcal F}_{t}), \tilde \P )$. As in Proposition 21 of \cite{DFPR13} we obtain that 
\begin{gather*}  
 \begin{array}{l} 
V_t = e^{t A }x +\int_{0}^{t} e^{(  t-s)  A}B_0  (
V_{s})ds
 + \frac{1}{2}\int_{0}^{t} e^{(  t-s)  A} \partial_{\xi}(
V_{s}^2)ds 
+
\int_{0}^{t}e^{(  t-s)  A}d \tilde W_{s}, \;\; t \in [0,T],
 \end{array}
\end{gather*}
 $\tilde \P$-a.s.. 
 Thus $V$ is a mild solution on $[0,T]$ to \eqref{cheef} defined on $(
\Omega, {\mathcal F},$ $
 ({\mathcal F}_{t}), \tilde \P )$.

\section{Proof of Proposition \ref{prop:K}}\label{prop} 

 Thanks to \eqref{s33} it is enough to consider  $\kappa=0.$
We revisit Proposition 2.1 in \cite{DPD} for our process $z_{\alpha}$.  
 Recall
\begin{align*}
z_{\alpha}(t) &= \int_0^t e^{(t-s)(A - \alpha)} (-A)^{\gamma} \, dW(s), \quad t \geq 0, \, \alpha \geq 0, \\
Y(\sigma) &= \int_0^{\sigma} (\sigma - s)^{- \beta} e^{(\sigma - s) A} (-A)^{\gamma} \, dW(s), \quad \sigma \geq 0, \, \beta \in (0, 1).
\end{align*}
We perform the factorization method (\cite{DPZ14}, Section 5.3.1) and the same proof  as in \cite{DPD}. Note, though, that in their proof, they introduce a parameter $\gamma \in (0, 1)$. To avoid the collision of notation, we write $\eta \in (0, 1)$ instead of their $\gamma$ and we write $\gamma$ that plays the role of ``our'' $\gamma \in [0, 1/4)$ in $(-A)^{\gamma}$.

We also achieve their equation (2.5)
$$
\| z_{\alpha}(t) \|_{L^p(\Lambda)} \leq c(\eta, \beta, m) \left( \alpha^{- \eta} + \alpha^{- \beta + \frac{1}{2m}} \right) \left( \int_0^t e^{- \lambda_p (t - \sigma)} \| Y(\sigma) \|_{L^p(\Lambda)}^{2m} \, d\sigma \right)^{\frac{1}{2m}},
$$
provided that conditions (2.3) and (2.4) hold:
\begin{equation*}
\beta > \eta + \frac{1}{2m}, \quad \beta > \frac{1}{2m}, \tag{*}
\end{equation*}
which we also must fulfill.
Regarding the moments of $Y$, we have
$$
Y(\sigma, \xi) = \sum_{k=1}^{\infty} \int_0^{\sigma} (\sigma - s)^{- \beta} e^{- \alpha_k (\sigma - s)} \alpha_k^{\gamma} e_k(\xi) \, d\beta_k(s) \;\; \text{and so}
$$
$$
\E |Y(\sigma, \xi)|^2 = \sum_{k=1}^{\infty} \int_0^{\sigma} s^{- 2 \beta} e^{- 2 \alpha_k s} \alpha_k^{2 \gamma} e_k^2(\xi) \, ds \lesssim \sum_{k=1}^{\infty} \frac{1}{\alpha_k^{1 - 2 \beta - 2 \gamma}}.
$$
The last sum is finite if and only if
\begin{equation*}
\beta < \frac{1}{4} - \gamma. \tag{**}
\end{equation*}
By the Gaussianity of $Y$, it follows that
\begin{equation} \label{eq:Y moments}
\E \| Y(\sigma) \|_{L^p(\Lambda)}^k \leq c(p, k),
\end{equation}
for all $p, k \geq 1$ and all $\sigma \geq 0$.
We continue in the proof as in \cite{DPD} and we obtain
$$
\| z_{\alpha}(t) \|_{L^p(\Lambda)} \leq c(\eta, \beta, m) \left( \alpha^{- \eta} + \alpha^{- \beta + \frac{1}{2m}} \right) \left( \int_0^{\infty} (1 + \sigma^2)^{-1} \| Y(\sigma) \|_{L^p(\Lambda)}^{4m} \, d\sigma \right)^{\frac{1}{4m}}.
$$
To prove the result, we need to satisfy conditions (*) and (**). Therefore, we choose $\eta, \beta, m$ such that
$$
\frac{1}{2m} < \min \{ \delta, \eps \}, \quad \eta := \frac{1}{4} - \eps - \gamma, \quad \frac{1}{4} - \gamma > \beta > \frac{1}{4} - \eps - \gamma + \frac{1}{2m}.
$$
This choice gives (*) and (**), but we also must satisfy that $\eta \in (0, 1)$ and $\beta \in (0, 1)$. That poses a condition
$$
\eps < \frac{1}{4} - \gamma,
$$
which we must assume. The final result
$$
\| z_{\alpha}(t) \|_{L^p(\Lambda)} \leq \alpha^{- \frac{1}{4} + \eps + \gamma} \left(1 + t^{\delta} \right) K_{\eps, \delta, p},
$$
holds with the random variable
\begin{equation}\label{eq:Kbalbla}
K_{\eps, \delta, p} = c(\eps, \delta) \left( \int_0^{\infty} (1 + \sigma^2)^{-1} \| Y(\sigma) \|_{L^p(\Lambda)}^{4m} \, d\sigma \right)^{\frac{1}{4m}}.
\end{equation}
The fact that
$$
\E K_{\eps, \delta, p}^k < \infty,
$$
for all $k \geq 1$, comes from \eqref{eq:Y moments}. The proof is complete. \qed

\appendix 
  
\section{}\label{section:appendix}

\subsection{The existence of the invariant measure}

The proof of the existence of the invariant measure $\mu_{\gamma}$ for the solution to \eqref{eq:Burgers} for $\gamma \in [0, 1/4)$ completely  follows Section 4 in \cite{DPDT}. Here we use the same  notation of \cite{DPDT}. 

Basically, wherever the authors in \cite{DPDT} work with the space $H^{1/4}$, we work with the space $H^{1/4 - \gamma}$, because our driving process $(-A)^{\gamma} \, W(t)$ is rougher. 


First, we consider the double-sided cylindrical process $W(t)$ on $\R$ by setting (extending)
$$
W(t) = V(-t), \, \quad t \leq 0,
$$
where $(V(t), \, t \geq 0)$ is another cylindrical $H$-valued Wiener process independent of $(W(t), \, t \geq 0)$.
For each $\lambda \geq 0$ we consider the solution $u_{\lambda}=(u_{\lambda}(t), t \geq - \lambda)$ to 
\begin{equation}
\begin{cases}
du_{\lambda}(t) &= \left( A u_{\lambda}(t) + \frac{1}{2} \partial_{\xi} \left( u_{\lambda}^2(t) \right) \right) dt + (-A)^{\gamma} \, dW(t), \\
u_{\lambda}(- \lambda) &= 0.
\end{cases}
\end{equation}
We introduce a modified stochastic convolution. For any $\alpha > 0$ we define
$$
W_A^{\alpha}(t) = \int_{- \infty}^t e^{(t-s) (A - \alpha)} (-A)^{\gamma} \, dW(s).
$$
$W_A^{\alpha}$ is the mild solution to the linear equation
\begin{equation}
\begin{cases}
dz(t) &= (Az(t) - \alpha z(t)) dt + (-A)^{\gamma} \, dW(t), \\
z(0) &= z_0,
\end{cases}
\end{equation}
where
$
z_0 = \int_{- \infty}^0 e^{-s (A - \alpha)} (-A)^{\gamma} \, dW(s).
$
It is possible 
to check that $(W_A^{\alpha}(t), t \in \R)$ is a stationary process. (It is Gaussian with zero mean and $\Tr Q_t = \sum_{k=1}^{\infty} \frac{(\pi k)^{4 \gamma}}{2(\pi^2 k^2 + \alpha)}$ independent of $t$; see also Lemma \ref{lemma:LDPDT} below.)
Define
$$
v_{\lambda}^{\alpha}(t) = u_{\lambda}(t) - W_A^{\alpha}(t), \quad t \geq - \lambda.
$$
Then $(v_{\lambda}^{\alpha}(t), t \geq - \lambda)$ is the mild solution to the semilinear equation
\begin{equation} \label{eq:vlambda}
\begin{cases}
\frac{d}{dt} v_{\lambda}^{\alpha}(t) &= \left( A v_{\lambda}^{\alpha}(t) + \frac{1}{2} \partial_{\xi} \left( v_{\lambda}^{\alpha}(t) + W_A^{\alpha}(t) \right)^2 \right) + \alpha W_A^{\alpha}(t), \\
v_{\lambda}^{\alpha}(- \lambda) &= - W_A^{\alpha}(- \lambda).
\end{cases}
\end{equation}
Now we modify Lemma 4.1 and Theorem 4.1 from \cite{DPDT}  because of our driving process $(-A)^{\gamma} \, W(t)$.

\begin{lemma} \label{lemma:LDPDT}
For any $\eps > 0$ and $\sigma \in [0, 1/4 - \gamma)$, there exists $\alpha$ (depending on $\eps$ and $\sigma$) such that
$$
\E \| (-A)^{\sigma} W_A^{\alpha}(t) \|^2 < \eps,\;\; \text{for all $t \in \R$.}
$$
\end{lemma}
\begin{proof}
We have
\begin{align*}
\E \| (-A)^{\sigma} W_A^{\alpha}(t) \|^2 &= \sum_{k=1}^{\infty} \int_{- \infty}^t (\pi k)^{4 \sigma} (\pi k)^{4 \gamma} e^{-2 (\pi^2 k^2 + \alpha)(t-s)} \, ds \\
&= \sum_{k=1}^{\infty} \frac{(\pi k)^{4 (\sigma + \gamma)}}{2(\pi^2 k^2 + \alpha)} < \infty,
\end{align*}
because $\gamma \in [0, 1/4)$ and $\sigma \in [0, 1/4 - \gamma)$.
Let $N$ be such that
$$
\sum_{k=N+1}^{\infty} (\pi k)^{4(\sigma + \gamma) - 2} < \frac{\eps}{2}.
$$
Now the assertion follows by choosing 
$\alpha$ 
such that
$
\sum_{k=1}^N \frac{(\pi k)^{4 (\sigma + \gamma)}}{2(\pi^2 k^2 + \alpha)} < \frac{\eps}{2}.
$
\end{proof}

\begin{theorem}
There exists an invariant measure $\mu_{\gamma}$ for the stochastic Burgers equation \eqref{eq:Burgers}.
\end{theorem}
\begin{proof}
Similarly to \cite{DPDT}, we will show that $\{ u_{\lambda}(0) \}_{\lambda \geq 0}$ is bounded in probability in $H^{1/4 - \gamma}$. Since the embedding $H^{1/4 - \gamma} \subset H$ is compact (that is by Rellich-Kondrachov theorem), the family of $\{ \Law \left( u_{\lambda}(0) \right) \}_{\lambda \geq 0}$ is tight, so there will exist an invariant measure (cf. Chapter 11 in \cite{DPZ14}).

Since $(W_A^{\alpha}(t), t \in \R)$ is bounded in probability in $H^{1/4 - \gamma}$, it is sufficient to show that $\{ v_{\lambda}^{\alpha}(0) \}_{\lambda \geq 0}$ is bounded in probability in $H^{1/4 - \gamma}$ for some $\alpha>0$. 

We first choose $\alpha$ such that, for all $t \in \R$
\begin{equation} \label{eq:WAalpha}
\E \| W_A^{\alpha}(t) \|_{L^4(\Lambda)}^{8/3} \leq \frac{\pi^2}{8C},
\end{equation}
where the constant $C > 0$ will be defined (or rather will appear) later. This is possible, because H\"older inequality gives
$$
\E \| W_A^{\alpha}(t) \|_{L^4(\Lambda)}^{8/3} \leq \left( \E \| W_A^{\alpha}(t) \|_{L^{\infty}(\Lambda)}^4 \right)^{1/3} \left( \E \| W_A^{\alpha}(t) \|^2 \right)^{2/3}
$$
and the first term on the right-hand side is bounded and the second term can be made arbitrarily small by Lemma \ref{lemma:LDPDT} above. 

We multiply \eqref{eq:vlambda} by $v_{\lambda}^{\alpha}(t)$ and integrate over $\Lambda$. We make estimates as in \cite{DPDT}, using also the Poincare inequality, so we eventually arrive at
$$
\frac{d}{dt} \| v_{\lambda}^{\alpha}(t) \|^2 + \pi^2 \| v_{\lambda}^{\alpha}(t) \|^2 \leq 2C \| W_A^{\alpha}(t) \|_{L^4(\Lambda)}^{8/3} \| v_{\lambda}^{\alpha}(t) \|^2 + \| W_A^{\alpha}(t) \|_{L^4(\Lambda)}^4 + \frac{2 \alpha^2}{\pi^2} \| W_A^{\alpha}(t) |^2. 
$$
By the Gronwall lemma, it follows
\begin{align}
\| v_{\lambda}^{\alpha}(0)|^2 &\leq e^{- \pi^2 \lambda + 2C \int_{-\lambda}^0 \| W_A^{\alpha}(s) \|_{L^4(\Lambda)}^{8/3} \, ds} \| W_A^{\alpha}(-\lambda) \|^2 \notag \\
&\phantom{=}+ \int_{-\lambda}^0 \left( \| W_A^{\alpha}(t) \|_{L^4(\Lambda)}^4 + \frac{2 \alpha^2}{\pi^2} \| W_A^{\alpha}(t) \|^2 \right) e^{\pi^2 s + 2C \int_s^0 \| W_A^{\alpha}(\sigma) \|_{L^4(\Lambda)}^{8/3} \, d\sigma} \, ds. \label{eq:vlambda0}
\end{align}
By stationarity of the process $(W_A^{\alpha}(t), t \in \R)$ and the ergodic theorem, we obtain as $\lambda \rightarrow \infty$
$$
\frac{1}{\lambda} \int_{-\lambda}^0 \| W_A^{\alpha}(s) \|_{L^4(\Lambda)}^{8/3} \, ds \rightarrow \E \| W_A^{\alpha}(0) \|_{L^4(\Lambda)}^{8/3}, \quad \P-a.s.
$$
From \eqref{eq:WAalpha}, we deduce that there exists a random variable $\lambda_0(\omega)$ such that
$$
e^{- \pi^2 \lambda + 2C \int_{-\lambda}^0 \| W_A^{\alpha}(s) \|_{L^4(\Lambda)}^{8/3} \, ds} \leq e^{- \frac{\pi^2}{2} \lambda}, \quad \text{for } \lambda \geq \lambda_0, \quad \P-a.s.
$$
Since $\| W_A^{\alpha}(s) \|^2$ and $\| W_A^{\alpha}(s) \|_{L^4(\Lambda)}^4$ have at most polynomial growth when $s \rightarrow -\infty$, for almost every $\omega \in \Omega$, the right-hand side of \eqref{eq:vlambda0} is bounded almost surely, i.e., there exists a random variable $R_1(\omega)$ such that
$$
\| v_{\lambda}^{\alpha}(0) \| \leq R_1(\omega), \quad \P-a.s.
$$
That proves that $\{ v_{\lambda}^{\alpha}(0) \}_{\lambda \geq 0}$ is bounded almost surely (and in probability) in $H$. 

Let us write $v_{\lambda}^{\alpha}(0)$ using the formulation of mild solution to \eqref{eq:vlambda}
$$
v_{\lambda}^{\alpha}(0) = - e^{\lambda A} W_A^{\alpha}(- \lambda) + \frac{1}{2} \int_{-\lambda}^0 e^{-sA} \partial_{\xi} (v_{\lambda}^{\alpha}(s) + W_A^{\alpha}(s))^2 \, ds + \alpha \int_{-\lambda}^0 e^{-sA} W_A^{\alpha}(s) \, ds.
$$
We deduce
\begin{align*}
\| (-A)^{1/8 - \gamma/2} v_{\lambda}^{\alpha}(0) \| &\leq \| e^{\lambda A} (-A)^{1/8 - \gamma/2} W_A^{\alpha}(- \lambda) \| \\
&\phantom{=}+ \frac{1}{2} \int_{-\lambda}^0 \left\| (-A)^{1/8 - \gamma/2} e^{-sA} \partial_{\xi} (v_{\lambda}^{\alpha}(s) + W_A^{\alpha}(s))^2 \right\| \, ds \\
&\phantom{=}+ \alpha \int_{-\lambda}^0 \left\|(-A)^{1/8 - \gamma/2} e^{-sA} W_A^{\alpha}(s) \right\| \, ds.
\end{align*}
We have upper bounds
\begin{align*}
&\left\|(-A)^{1/8 - \gamma/2} e^{-sA} \partial_{\xi} (v_{\lambda}^{\alpha}(s) + W_A^{\alpha}(s))^2 \right\| \\
&\quad \leq C_1 \left( (-s)^{-7/8 + \gamma/2} + 1 \right) e^{\pi^2 s} \left\|v_{\lambda}^{\alpha}(s) + W_A^{\alpha}(s) \right\|^2, \\
&\left\|(-A)^{1/8 - \gamma/2} e^{-sA} W_A^{\alpha}(s) \right\| \leq e^{\pi^2 s} \left\|(-A)^{1/8 - \gamma/2} W_A^{\alpha}(s) \right\|,
\end{align*} 
for any $0 > s \geq - \lambda$. Note that, $\P$-a.s., for any $T>0,$ $W_A^{\alpha} \in C([0,T]; H^{1/2 - 2 \gamma - \eps})$ for any $\eps > 0$ (see Theorem 5.15 in \cite{DPZ14}),  so the process $\left\|(-A)^{1/8 - \gamma/2} W_A^{\alpha}(s) \right\|$ is well defined and has at most polynomial growth as $s \rightarrow - \infty$ for almost all $\omega \in \Omega$. Therefore, we deduce the existence of a~random variable $R_2(\omega)$ such that
$$
\| (-A)^{1/8 - \gamma/2} v_{\lambda}^{\alpha}(0) \| \leq R_2(\omega), \quad \P-a.s.  
$$ 
Hence $\{ v_{\lambda}^{\alpha}(0) \}_{\lambda \geq 0}$ (and therefore  $\{ u_{\lambda}(0) \}_{\lambda \geq 0}$) is bounded almost surely in $H^{1/4 - \gamma}$. Since almost sure boundedness implies boundedness in probability and the embedding  $H^{1/4 - \gamma} \subset H$ is compact, the family of $\{ \Law \left( u_{\lambda}(0) \right) \}_{\lambda \geq 0}$ is tight and there exists an invariant measure $\mu_{\gamma}$ for the stochastic Burgers equation \eqref{eq:Burgers}.
\end{proof}

\subsection{Irreducibility}

\begin{proposition}
The transition Markov semigroup $(P_t)$ corresponding to the solution $X$ of equation \eqref{eq:Burgers} is irreducible.
\end{proposition}
\begin{proof}
 We follow the idea of the proof of Proposition 14.4.3 of \cite{DPZ96} which considers the case  $\gamma =0$ (see also Section 5.5 in \cite{DP}). We use  a notation similar to  \cite{DPZ96}.

Let $T>0$. By the density of $ H^1= H^1_0$ in $H= L^2(\Lambda)$ it is enough to prove that for any $a \in H$, $b \in H^1= H^1_0$, $\epsilon'>0,$ we have
\begin{equation}\label{irr1}
\P (\|X_T^{a} - b\| < \epsilon') >0 
\end{equation}
 (recall that $\| \cdot\| = \| \cdot\|_H$). 
 Arguing as in  \cite{DPZ96} we may assume in addition that $a \in H^1$. 
 Let $\epsilon>0$. One can prove that there exists a control function $v \in L^2(0,T;H)$ such that 
 \begin{gather} \label{s22}
y(t, \xi) \!= e^{tA} a + \frac{1}{2} \int_0^t e^{(t-s)A} \partial_{\xi} (y^2(s, \cdot))(\xi)  ds + \! \int_0^t (-A)^{\gamma} e^{(t-s)A} v(s, \cdot)(\xi)  ds, 
\end{gather}
 $\xi \in \Lambda$, verifies $\| y(T) - b\| < \epsilon$ (approximate controllability).  To this purpose we first recall that there exists $u \in L^2(0,T;H)$ such that 
 \begin{gather*}  
\tilde z(t) = e^{tA} a + \int_0^t (-A)^{\gamma} e^{(t-s)A} u(s) \, ds, 
\end{gather*}
 verifies $\| \tilde z(T)- b\|< \epsilon$. Then since $\tilde z \in C([0,T]; H^1)$ we define 
 $$
 v(t, \xi) = -\frac{1}{2}(-A)^{-\gamma} \partial_{\xi}(\tilde z^2(t,\cdot))(\xi) + u(t, \xi), \;\; t \in (0,T),\; \xi \in \Lambda = (0,1). 
 $$
It is clear  that $\tilde z$ solve the problem \eqref{s22}. 
 We consider
 \begin{gather*}
W_A(t) =\int_0^t (-A)^{\gamma} e^{A (t-s)}  \, dW(s),\;\; \; V_A(t) =\int_0^t (-A)^{\gamma} e^{A (t-s)}  \, v(s) ds.
\end{gather*}
As usual we argue $\omega$ by $\omega$. Assume that $\sup_{t \in [0,T]} \| W_A(t)\|_{C(\bar \Lambda)} \le \tilde \gamma$ ($\tilde \gamma$ will be fixed later).
 Comparing $X_t^{a}$ with  the solution $y$ of \eqref{s22} we get
 \begin{gather*}
\| X_t^{a} - y(t)\| \le \frac{c}{2} \int_0^t (t-s)^{-3/4}\, \| X_s^{a} - y(s)\|ds
 + \|W_A(t) - V_A(t)\|,
\end{gather*}
 where $c$ only depends  on $\tilde \gamma$. 
 By a generalized  Gronwall lemma, we deduce
 \begin{gather*}  
\sup_{t \le T}\| X_t^{a} - y(t)\| \le c_T \sup_{t \le T}\|W_A(t) - V_A(t)\|.
\end{gather*}
 Now as in \cite{DPZ96} we use that the support of the process $W_A(\cdot)$ in $C([0,T]; H)$ is the closure of the set of all functions
 \begin{gather*}
 \int_0^t  (-A)^{\gamma} e^{A (t-s)}w(s)ds,\;\;\;\; w \in L^2(0,T; H).
\end{gather*}
 It follows that  $\P(\sup_{t \le T}\|W_A(t) - V_A(t)\|< \epsilon ) >0.$ 
 
 We fix $\tilde \gamma = \epsilon + \sup_{t \le T}\| V_A(t)\|$. Arguing as in page 280 of \cite{DPZ96} we obtain
 \begin{gather*}
\P \big(\sup_{t \le T}\| X_t^{a} - y(t)\| < \epsilon \, c_T \big ) 
\\ \ge \P \big (\sup_{t \le T} \|W_A(t) - V_A(t)\|< \epsilon,  \;\;\; \sup_{t \in [0,T]} \| W_A(t)\|_{C(\bar \Lambda)}  \le \tilde \gamma \big ) >0.
\end{gather*}
 Since $\| y(T) - b \| < \epsilon$ and $\epsilon$ can be arbitrarily chosen the previous inequality implies  assertion \eqref{irr1}.
 \end{proof}
 
\def\ciao2{
E. The proof of \cite{DP}, page 143 requires to consider a normal  r.v. with values in $C([0,T]; C([0,1]))$. This is more difficult to be justified than
\cite{DP96}. 

\cite{DP}, page 143. The fact they use periodic boundary conditions and $A = \Delta - I$ is not a problem. However, whenever the author use $(-A)^{- \gamma/2} u(t)$ for $\gamma > 1/2$, we use $(-A)^{\gamma} u(t)$ for $\gamma \in [0, 1/4)$. \\

Consider the following controlled system
\begin{equation}
\begin{cases}
y'(t) &= Ay(t) + \frac{1}{2} \partial_{\xi} (y^2(t)) + (-A)^{\gamma} u(t), \quad t \in [0,T], \\
y(0) &= x \in H,
\end{cases}
\end{equation}
where $\gamma \in [0, 1/4)$ and $u \in L^2([0,T]; H)$. (We will actually consider $u \in L^2([0,T]; H^{2 \gamma + 1/2 + \eps'})$ for the purposes below.) The mild solution of the above equation is
$$
y(t) = e^{tA} x + \frac{1}{2} \int_0^t e^{(t-s)A} \partial_{\xi} (y^2(s)) \, ds + \int_0^t (-A)^{\gamma} e^{(t-s)A} u(s) \, ds, \quad t \geq 0.
$$
Moreover, if $x \in C(\bar{\Lambda})$ then $y \in C([0,T]; C(\bar{\Lambda}))$. (This is because of the smooth\-ing property of the semigroup. Indeed, we have $\int_0^t e^{(t-s)A} u(s) \, ds \in H^s$ for $s \in [0,2)$.) \\

We adjust Proposition 5.15 in \cite{DP} in the following way.

\begin{proposition} \label{prop:irr1}
Let $T > 0$, $x_0, x_1 \in H$, $\eps > 0$. Then there exists $u \in C([0,T]; H^{2 \gamma + 1/2 + \eps'})$ for some $\eps' > 0$ such that
$$
\| y(T, x_0; u) - x_1 \| \leq \eps.
$$
\end{proposition}
\begin{proof}
We follow the setting and the proof of Proposition 5.15. The function $\beta_{z_0, z_1}(t) \in C(\bar{\Lambda})$, so we must adjust the space for our function $u(t)$ accordingly to the requirement
\begin{itemize}
\item[(ii)] $\| \beta_{z_0, z_1}(t) - (-A)^{\gamma} u(t) \| \leq c \eps, \quad t \in [0,T]$.
\end{itemize}
The reason why we select a function $u$ from the space $C([0,T]; H^{2 \gamma + 1/2 + \eps'})$ is such that $(-A)^{\gamma} u(t) \in H^{1/2 + \eps'} \subset C(\bar{\Lambda}) \subset H$. Note that $C(\bar{\Lambda})$ is dense in $H$, so such function for (ii) can be selected. (In a more abstract setting of Chapter 3 in \cite{DP}, we use $U = H^{2 \gamma + 1/2 + \eps'}$, $B(U) = H^{1/2 + \eps'} \subset C(\bar{\Lambda}) \subset H$.)

The proof is then performed and finished as the proof of Proposition 5.15.
\end{proof}

\begin{proposition}
The transition semigroup $(P_t, t \geq 0)$ corresponding to the solution $X$ of equation \eqref{eq:Burgers} is irreducible.
\end{proposition}
\begin{proof}
For given $\eps, \eps' > 0$, we use the above Proposition \ref{prop:irr1} to obtain a function $u \in C([0,T]; H^{2 \gamma + 1/2 + \eps'})$ and we repeat the proof of Theorem 5.16 in \cite{DP}. The last line of the proof will use the fact that our process of the stochastic convolution $(z(t), t \geq 0)$ (that is denoted as $(W_A, t \geq 0)$ in \cite{DP}) is full in $C([0,T]; C(\bar{\Lambda}))$, since $\Ker (-A)^{2 \gamma} = \{ 0 \}$. (Cf. Exercise 2.16.)
\end{proof}
 } 
  
\section{The well-posedness of the equation \eqref{eq:Burgers}}\label{ppp2}
 
{\bf Proof of Proposition \ref{cia1}} Taking into account Remark \ref{sys}
   we know that   $z \in C([0,T]; C(\bar{\Lambda}))$, $\P$-a.s.  
 Next, we follow the approach of \cite{G} (cf. also \cite{DPZ96} and \cite{JP}).
 
For any $n \geq 1$, we introduce the orthogonal projection $\tilde \pi_n: H \rightarrow B(0, n)$, with $B(0, n) = \{ u \in H; \| u \| \leq n \}$. One can show that the equation (for $t \in [0,T]$, $\P$-a.s.) 
\begin{equation} \label{eq:Burgers with n}
X_n(t) = e^{At} x + \frac{1}{2} \int_0^t e^{A (t-s)} \partial_\xi \left( [\tilde \pi_n (X_n(s))]^2 \right) \, ds + \int_0^t e^{A (t-s)} (-A)^{\gamma} \, dW_s,
\end{equation} 
has a unique solution on a small time interval $[0,R]$, $0 < R \leq T$ by the contraction principle (we do not stress the dependence on $x$ in  the solution). The time $R = R(n, T)$ is deterministic.
Solutions of \eqref{eq:Burgers with n} can be defined for all $t \geq 0$. We define the stopping times
$$
\tau_n := \inf \{ t \geq 0; \| X_n(t) \| \geq n \}, \quad n \geq 1.
$$
Since $X_n(t) = X_m(t)$, $m \geq n$, $t \leq \tau_n$, one can set $X_t = X_n(t)$, $t \leq \tau_n$ and define a solution of \eqref{eq:Burgers with n} on $[0, \tau_{\infty})$. We obtain a global mild solution by proving that $\tau_{\infty} = \infty$, i.e., $\tau_n \rightarrow \infty$, $\P$-a.s. by means of {\it a-priori} estimates.
  Denoting
$$
Y_t := X_t - z(t),
$$
we can follow the proof of Proposition 2.9 in \cite{JP} to obtain
$$
\frac{1}{2} \frac{d}{dt} \| Y_t \|^2 + \| Y_t \|_1^2 \leq \frac{3}{4} \| Y_t \|_1^2 + 2 \| z(t) \|_{C(\bar{\Lambda})}^2 \| Y_t \|^2 + \| z(t) \|_{C(\bar{\Lambda})}^4
$$
(we consider a regular approximation of $z(t)$ and the corresponding regular solution $Y_t$;   note that the final obtained estimates hold with $z(t)$, i.e., they do not require the presence of the regular approximation). 
Gronwall lemma provides
$$
\| Y_t \|^2 \leq e^{4 \mu T} \left( \| x \|^2 + 2 \mu^2 \right),
$$
for $t \in [0,T]$ and $\mu = \sup_{t \in [0,T]} \| z(t) \|_{C(\bar{\Lambda})}^2$. Obtaining such {\it a-priori} estimate finishes the proof as in  page 263 of \cite{DPZ96}. \qed 
 
\section{Proof of Proposition \ref{continuo}}   \label{app1} 
 
\noindent The proof  is inspired by Section 5.5 in \cite{DPZ14}
 and requires some preliminaries 
on  the random field $z_{\alpha}$  given in  \eqref{zz}.
 
Recall that  the process $z_{\alpha}$, $\alpha \ge 0,$ 
 is the unique mild solution to the equation
$$
dz_{\alpha}(t) = (A - \alpha) z_{\alpha}(t) \, dt + (-A)^{\gamma} \, dW(t), \quad z_{\alpha}(0) = 0.
$$
Clearly, $z_0(t) = z(t)$. 
Taking into account Remark \ref{sys}  we can adapt the arguments of Section 5.5.1 in \cite{DPZ14}. In particular 
the operator $A$ satisfies assumptions (5.39)--(5.41) from \cite{DPZ14}:
\begin{equation}
A e_k = - \alpha_k e_k=  - \pi^2 k^2 e_k, \quad k \geq 1, 
\tag{5.39}
\end{equation}
and, using $\lambda_k = \alpha_k^{2\gamma} = (\pi^2 k^2)^{2 \gamma}$,
\begin{equation}
\alpha_k \geq L > 0, \quad \sum_{k=1}^{\infty} \frac{\lambda_k}{\alpha_k} = \sum_{k=1}^{\infty} \frac{1}{\alpha_k^{1 - 2 \gamma}} \cong \sum_{k=1}^{\infty} \frac{1}{k^{2 - 4 \gamma}} < \infty, \tag{5.40}
\end{equation}
\begin{equation}
(e_k) \subset C([0,1]), \quad |e_k(\xi)| \leq C, \quad |\nabla e_k(\xi)| \leq C \alpha_k^{1/2}, \quad \xi \in (0,1), \, k \geq 1, \tag{5.41}
\end{equation}
for some constants $C, L > 0$ and $\gamma \in [0, 1/4)$. $A$ also satisfies assumption (5.42) with  $\eta = \frac{1}{2} - 2 \gamma - \eps \in (0,1)$ and $\eps$ small enough:
\begin{equation}
\sum_{k=1}^{\infty} \frac{\lambda_k}{\alpha_k^{1 - \eta}} = \sum_{k=1}^{\infty} \frac{1}{\alpha_k^{1 - 2 \gamma - \eta}} \cong \sum_{k=1}^{\infty} \frac{1}{k^{2 - 4 \gamma - 2 \eta}} < \infty. \tag{5.42}
\end{equation}
Now we use Theorem 5.22 (see also Lemma 5.21) in \cite{DPZ14} for the ``shifted'' operator $A_{\alpha} := A - \alpha I$ with $\alpha \geq 0$. Note that the eigensystem for $A_{\alpha}$ satisfies
$$
A_{\alpha} e_k = (- \alpha_k - \alpha) e_k, \quad k \geq 1,
$$
and fulfills the above assumptions (5.39)--(5.41). The assumption (5.42) is also fulfilled with $\eta = \frac{1}{2} - 2 \gamma - \eps$, since
$$
\sum_{k=1}^{\infty} \frac{\lambda_k}{(\alpha_k + \alpha)^{1 - \eta}} = \sum_{k=1}^{\infty} \frac{\alpha_k^{2 \gamma}}{(\alpha_k + \alpha)^{1 - \eta}} \cong \sum_{k=1}^{\infty} \frac{1}{k^{2 - 4 \gamma - 2 \eta}} < \infty.
$$
For the stochastic convolution $z_{\alpha}(t, \xi)$, we have the following representation
$$
z_{\alpha}(t, \xi) = \sum_{k=1}^{\infty} \alpha_k^{\gamma} e_k(\xi) \int_0^t e^{- (\alpha_k + \alpha) (t-s)} \, d\beta_k(s), \quad \xi \in [0,1], \, t \in [0,T],
$$
where $(\beta_k)$ are independent real Wiener processes.\\
 
\noindent { \bf  Proof of Proposition \ref{continuo}.} Thanks to \eqref{s33} it is enough to provide the proof when $\kappa=0.$
  Consider $0 \leq s \leq t \leq T$, $\alpha, \alpha' \geq 0$ and $\xi, \zeta \in [0,1]=\bar \Lambda$.  We will suitably  apply the Kolmogorov test. 

It is well-known that  that $z_{\alpha}(t, \xi) - z_{\alpha'}(s, \zeta)$ is Gaussian. Indeed, both $z_{\alpha}(t, \xi)$ and $z_{\alpha'}(s, \zeta)$ have a representation based on  series of  independent stochastic integrals. We have
\begin{align*}
\E |z_{\alpha}(t, \xi) - z_{\alpha'}(s, \zeta)|^2 &\leq 3 ( \E |z_{\alpha}(t, \xi) - z_{\alpha}(s, \xi)|^2 + \E |z_{\alpha}(s, \xi) - z_{\alpha}(s, \zeta)|^2 \\
&\phantom{=}+ \E |z_{\alpha}(s, \zeta) - z_{\alpha'}(s, \zeta)|^2)
\end{align*}
and we will find upper bounds for the three individual terms. 

(a). For the fist term, concerning continuity in $t$,  we follow the proof of Lemma 5.21 in \cite{DPZ14}.
Fix $0 < s < t \leq T$, $\xi \in [0,1]$  and $\alpha \geq 0$. We have
\begin{gather*}
\E \left| z_{\alpha}(t, \xi) - z_{\alpha}(s, \xi) \right|^2 = \sum_{k=1}^{\infty} \alpha_k^{2 \gamma} |e_k(\xi)|^2 \int_s^t e^{-2 (\alpha_k + \alpha)(t - \sigma)} \, d\sigma \\
+ \sum_{k=1}^{\infty} \alpha_k^{2 \gamma} |e_k(\xi)|^2 \int_0^s \left( e^{- (\alpha_k + \alpha)(t - \sigma)} - e^{- (\alpha_k + \alpha)(s - \sigma)} \right)^2 \, d\sigma 
=: I_1(t, s, \xi, \alpha) + I_2(t, s, \xi, \alpha).
\end{gather*}
As in the proof of Lemma 5.21, we arrive at
\begin{align*}
I_1(t, s, \xi, \alpha) &\lesssim |t-s|^{\eta} \sum_{k=1}^{\infty} \frac{\alpha_k^{2 \gamma}}{(\alpha_k + \alpha)^{1 - \eta}}, \;\;\;\;
I_2(t, s, \xi, \alpha) &\lesssim |t-s|^{\eta} \sum_{k=1}^{\infty} \frac{\alpha_k^{2 \gamma}}{(\alpha_k + \alpha)^{1 - \eta}}. 
\end{align*}
From this, we see that we can choose $\eta$ like $\eta = \frac{1}{2} - 2 \gamma - \eps$ for some small $\eps > 0$. 
 
(b). For the second term, concerning continuity in $\xi$, fix $\xi, \zeta \in [0,1]$, $0 < s \leq T$ and $\alpha \geq 0$. We have
\begin{gather*}
\E \left| z_{\alpha}(s, \xi) - z_{\alpha}(s, \zeta) \right|^2 
= \sum_{k=1}^{\infty} \alpha_k^{2 \gamma} |e_k(\xi) - e_k(\zeta)|^2 \int_0^s e^{-2 (\alpha_k + \alpha)(s - \sigma)} \, d\sigma 
\\
= \sum_{k=1}^{\infty} \alpha_k^{2 \gamma} |e_k(\xi) - e_k(\zeta)|^2 \frac{1 - e^{-2(\alpha_k - \alpha) s}}{2 (\alpha_k + \alpha)} 
\lesssim |\xi - \zeta|^{2 \eta} \sum_{k=1}^{\infty} \frac{\alpha_k^{2 \gamma + \eta}}{\alpha_k + \alpha},
\end{gather*}  
where we used that $|e_k(\xi) - e_k(\zeta)| \leq C 2^{1 - \eta} \alpha_k^{\eta/2} |\xi - \zeta|^{\eta}$ for $k \in \N$, $\eta \in [0,1]$ and some constant $C > 0$. From the last line of the above calculation, we see that we can again take $\eta = \frac{1}{2} - 2 \gamma - \eps$ for some small $\eps > 0$. 

(c). For the third term, concerning  the continuity in $\alpha$, we fix $\alpha, \alpha' \geq 0$, $0 < s \leq T$ and $\zeta \in [0,1]$. We have
\begin{align}
&\E \left| z_{\alpha}(s, \zeta) - z_{\alpha'}(s, \zeta) \right|^2 \notag 
\\
&\lesssim \sum_{k=1}^{\infty} \alpha_k^{2 \gamma} \int_0^T \left( e^{- (\alpha_k   + \alpha) \sigma} - e^{- (\alpha_k + \alpha') \sigma} \right)^2 \, d\sigma  \notag
= \sum_{k=1}^{\infty} \alpha_k^{2 \gamma} \int_0^T \Big[e^{- \alpha_k \sigma} \underbrace{(e^{- \alpha \sigma} - e^{- \alpha' \sigma})}_{\leq |\alpha - \alpha'|^{} \sigma^{}} \Big]^2 \, d\sigma 
\\
&
\leq |\alpha - \alpha'|^{2 } T^{2 } \sum_{k=1}^{\infty} \alpha_k^{2 \gamma} \int_0^T e^{-2 \alpha_k \sigma} \, d\sigma \notag 
\leq \frac{1}{2} |\alpha - \alpha'|^{2 } T^{2 } \sum_{k=1}^{\infty} \frac{1}{\alpha_k^{1 - 2 \gamma}}. \notag 
\end{align}
In total, the individual terms from parts (a), (b), (c) satisfy the following
\begin{align*}
\E |z_{\alpha}(t, \xi) - z_{\alpha}(s, \xi)|^2 &\leq C |t-s|^{\eta} \sum_{k=1}^{\infty} \frac{\alpha_k^{2 \gamma}}{(\alpha_k + \alpha)^{1-\eta}}, \quad \eta = \frac{1}{2} - 2 \gamma - \eps, \\
\E |z_{\alpha}(s, \xi) - z_{\alpha}(s, \zeta)|^2 &\leq C |\xi - \zeta|^{2 \eta} \sum_{k=1}^{\infty} \frac{\alpha_k^{2 \gamma + \eta}}{\alpha_k + \alpha}, \quad \eta = \frac{1}{2} - 2 \gamma - \eps, \\
\E |z_{\alpha}(s, \zeta) - z_{\alpha'}(s, \zeta)|^2 &\leq C |\alpha - \alpha'|^2 T^2 \sum_{k=1}^{\infty} \frac{1}{\alpha_k^{1 - 2 \gamma}},
\end{align*}
for some constant $C > 0$  (this is independent  of $\alpha, \alpha', s,t,\xi, \zeta$ and  may later change from line to line).  Therefore
$$
\E |z_{\alpha}(t, \xi) - z_{\alpha'}(s, \zeta)|^2 \leq C \left( |t-s|^{\eta} + |\xi - \zeta|^{2 \eta} + |\alpha - \alpha'|^2 \right), 
$$
for $\eta = 1/2 - 2 \gamma - \eps$. Choose $m \in \N$ such that $m \eta = m \left( 1/2 - 2 \gamma - \eps \right) > 1$. By Gaussianity of $z_{\alpha}$, we have  
\begin{align*}
\E |z_{\alpha}(t, \xi) - z_{\alpha'}(s, \zeta)|^{2m} &\leq C \left( \E |z_{\alpha}(t, \xi) - z_{\alpha'}(s, \zeta)|^2 \right)^m \\
&\leq C \left( |t-s|^{m \eta} + |\xi - \zeta|^{2m \eta} + |\alpha - \alpha'|^{2m} \right).
\end{align*}
The exponents $m \eta$, $2m \eta$ and $2m$ are bigger than $1$. Therefore we can use the Kolmogorov test (Theorem 3.5 in \cite{DPZ14}) and we obtain 
a version of $z_{\alpha} (t, \xi)$  denoted by $\bar  z_{\alpha}(t, \xi)$ which is continuous in all his variables. 

Note that  $\bar z_{\alpha}(t): \Omega \rightarrow C(\bar{\Lambda})$
is  continuous in $(\alpha, t) \in [0, \infty) \times [0,T]$. \qed 
 
\vskip 2mm {\small  \noindent {\bf Acknowledgments.}  The authors wish to
thank  M. Gubinelli for useful suggestions. This research activity was carried out as part of
the PRIN 2022 project ``Noise in fluid dynamics and related models''. 
 The first and third author are  members of 
``Gruppo Nazionale per l'Analisi Matematica, la Probabilità e le loro Applicazioni (GNAMPA)'', which
is part of INdAM.  
} 
  

\end{document}